\documentclass[10pt]{article}

\usepackage[colorlinks,linkcolor=blue,citecolor=blue]{hyperref}
\usepackage{amsthm}
\usepackage{multirow}
\usepackage{marginnote}
\usepackage{a4wide}
\usepackage{amssymb}
\usepackage{amsfonts}
\usepackage{amsmath}
\usepackage{mathrsfs}
\usepackage{tikz}
\usetikzlibrary{arrows,matrix}
\usetikzlibrary{positioning}
\usepackage{mdframed} 
\usepackage{lipsum} 
\usepackage{extarrows} 
\usepackage{color}
\usepackage{enumitem} 
\usepackage{tocloft} 
\usepackage{titlesec} 

\usepackage[all]{xy} 

\input xy
\xyoption{arrow} \xyoption{matrix}

\date{}

\newtheorem{proposition}{Proposition}[section]
\newtheorem{theorem}[proposition]{Theorem}
\newtheorem{lemma}[proposition]{Lemma}
\newtheorem{example}[proposition]{Example}
\newtheorem{definition}[proposition]{Definition}
\newtheorem{corollary}[proposition]{Corollary}

\def\Hom{{\rm Hom}}
\def\der{\partial }

\def\nFM0{{\nu }_{F,M_0}}
\def\nFN0{{\nu }_{F,N_0}}
\def\nGN0{{\nu }_{G,N_0}}

\def\N0{ {\bf N}_0 }

\def\t{\otimes}

\def\ra{\rightarrow}

\def\Xpm{X^{\pm }}

\def\s{\sigma}

\def\l1{{\lambda}_1}

\def\a{\alpha}
\def\a0{ {\alpha }_0}
\def\a1{ {\alpha }_1}

\def\l{\lambda}

\def\nFGM0{{\nu }_{F,G,M_0}}

\def\nFN0{{\nu}_{F,N_0}}

\def\sm{{\sigma}^m}

\def\sm1{{\sigma}^{-1}}

\def\smtp1{{\sigma}^{-t+1}}

\def\S1{S^{-1}}

\def\Xpm1{X^{\pm 1}_1}

\def\sPM1{{\sigma }^{\pm 1}}
\def\sMP1{{\sigma }^{\mp 1 }}

\def\b{\beta}

\def\di{{\rm d.ind}}

\def\L{\Lambda}

\def\CA{{\cal A}}

\def\Ytm1{Y^{t-1}}
\def\Yim1{Y^{i-1}}

\def\CM{{\cal M}}
\def\CN{{\cal N}}

\def\ass{{\rm ass}}

\def\dim{{\rm dim }}

\def\ker{ {\rm ker } }

\def\SL2Z{ {\rm SL}_2({\bf Z}) }

\def\Gp1{ G^{1 , 1 } }
\def\P11{ P^{-1 , 1 } }
\def\Pp1{ P^{1 , 1 } }

\def\CV{{\cal V}}

\def\nCLsr{{}^\nu\kern-2pt {\cal L}^{\sigma , \rho  }}
\def\nP{{}^\nu \kern-2pt P}
\def\nL{{}^\nu\kern-2pt L}
\def\nLL{{}^\nu\kern-2pt \Lambda}
\def\nPsr{{}^\nu\kern-2pt P^{\sigma , \rho  }}
\def\nLsr{{}^\nu\kern-2pt L^{\sigma , \rho  }}
\def\nuCL{{}^\nu\kern-2pt  {\cal L}}
\def\nCLsr{{}^\nu\kern-2pt {\cal L}^{\sigma , \rho  }}
\def\nCL1m{{}^\nu\kern-2pt {\cal L}^{-1 , 1  }}
\def\x1nu{x^\frac{1}{\nu}}
\def\xm1nu{x^{-\frac{1}{\nu}}}

\def\rad{{\rm rad}}

\def\CN{{\cal N}}
\def\ra{\rightarrow }

\def\CB{{\cal B}}

\def\CI{{\cal I}}

\def\CC{ {\cal C}}

\def\nAM0{{\nu }_{{\cal A},M_0}}
\def\nAN0{{\nu }_{{\cal A},N_0}}

\def\End{ {\rm End }}

\def\of{\overline{f}}

\def\ga{\mathfrak{a}}

\def\gn{\mathfrak{n}}

\def\gp{\mathfrak{p}}
\def\gq{\mathfrak{q}}

\def\SL{{\rm SL}}

\def\Spec{{\rm Spec}}

\def\Hom{{\rm Hom}}

\def\di!{\frac{\der^i}{i!}}
\def\dik!{\frac{\der^k_i}{k!}}

\def\gl{\mathfrak{l}}

\def\id{{\rm id}}

\def\N{\mathbb{N}}

\def\0{\overline{0}}
\def\1{\overline{1}}

\def\Ln1{\L_{n,\overline{1}}}

\def\a1{a_{\overline{1}}}

\def\S{\Sigma}

\def\vn1{\overrightarrow{n-1}}

\def\im{{\rm im}}

\def\gl{{\rm gl}}
\def\sl{{\rm sl}}

\def\mS{\mathbb{S}}
\def\mJ{\mathbb{J}}
\def\mI{\mathbb{I}}

\def\ann{{\rm ann}}
\def\lann{{\rm l.ann}}
\def\rann{{\rm r.ann}}

\def\bM{\overline{M}}

\def\K1{{\rm K}_1}

\def\hmI1{\widehat{\mI_1}}
\def\tmI1{\widetilde{\mI_1}}
\def\tmJ1{\widetilde{\mJ_1}}
\def\hB1{\widehat{B_1}}
\def\hCB1{\widehat{\CB_1}}

\def\Den{{\rm Den}}

\def\Ore{{\rm Ore}}

\def\Den{{\rm Den}}

\def\Loc{{\rm Loc}}

\def\maxDen{{\rm max.Den}}

\def\ga{\mathfrak{a}}

\def\tor{{\rm tor}}

\def \S{\mathcal{S}}

\def\sl2{\mathfrak{sl}_2}

\def\Irr{{\rm Irr}}

\def\sl2{\mathfrak{sl}_2}
\def\gl2{\mathfrak{gl}_2}

\def\Irr{{\rm Irr}}
\def\b1{\overline{1}}

\def\mR{\mathbb{R}}

\def\Irr{{\rm Irr}}

\def\gl{{\mathfrak{l}}}

\def\Mor{{\rm Mor}}
\def\Ob{{\rm Ob}}
\def\pr{{\rm pr}}

\makeatletter
\newenvironment{proof*}[1][\proofname]{\par
  \pushQED{\qed}%
  \normalfont \partopsep=\z@skip \topsep=\z@skip
  \trivlist
  \item[\hskip\labelsep
        \itshape
    #1\@addpunct{.}]\ignorespaces
}{%
  \popQED\endtrivlist\@endpefalse
}
\makeatother

\begin{document}

\author{V. V. \  Bavula  
}
\title{Embeddings of semiprime  rings  into  semisimple Artinian  rings}

\maketitle

\begin{abstract}
Goldie's Theorem implies that a semiprime left Goldie ring is embeddable into a semisimple Artinian ring. On the other hand, there are domains that are not embeddable into  division rings. A criterion for a semiprime ring being embeddable into a semisimple Artinian ring is given. Three types of embeddings of semiprime rings into semisimple Artinian rings are introduced and studied: the elementary, natural and collection embeddings. It is proven that for each embedding there is an elementary embedding which is smaller and  explicitly constructed. In particular, the  minimal embeddings are elementary ones. It is proven that a semiprime ring and all its localizations at regular left Ore sets have the `same' sets of elementary embeddings. A new  criterion for the left quotient ring of a ring being a semisimple Artinian ring is given in terms of the centre of the ring which is completely different from Goldie's Theorem (which is historically the first criterion, \cite{Goldie-PLMS-1960}). It is proven that the minimal embeddings in each  Morita equivalence class of embeddings  is a {\em non-empty} set and they are explicitly described. Descriptions of the minimal embeddings of semiprime left Goldie rings and all their localizations at regular left Ore sets are described.

$\noindent$

{\em Key Words: monomorphism, semiprime ring, prime ring, semiprimary ring, semisimple Artinian ring,  natural embedding, elementary embedding, localization, classical left  quotient ring, largest left quotient ring, Ore set, denominator set, order, maximal order. }

{\em Mathematics subject classification 2020:    16P20, 16U20,  16P50,  16P60, 16S85, 16U60.}    

{ \small \tableofcontents}
\end{abstract}


\section{Introduction} \label{INTR} 

{\bf Notation.} In this paper, module means a left module, and the following notation is fixed:
\begin{itemize}
\item  $R$ is a ring  with  1, $R^\times$ is its group of units, 
$Z(R)$ is the centre of $R$, 
 $\gn_R$ is the  prime  radical and $\rad (R)$ is the  radical  of $R$,  and $\min (R)$ is the  set  of minimal  primes of $R$;

\item $\Spec (R)$ is the prime spectrum and $\Spec_c (R)$ is the completely prime spectrum  of $R$; 

 \item $\Den_l(R, \ga )$ is the set of left denominator sets $S$ of $R$ with $\ass_l (S)=\ga$ where $\ga$ is an ideal of $R$ and $\ass_l (S):= \{r\in R\, | \, sr=0$ for some $s\in S\}$;
 
\item For $S\in \Den_l(R, \ga)$, $\min (R, S):=\{ \gp\in \min (R)\, | \, \gp\cap S=\emptyset\}$ and $\min (R, S, \id):=\{ \gp\in \min (R)\, | \, S^{-1}\gp S^{-1}R\neq S^{-1}R\}$;
 

\item   $\CC_R$  is the  set of  regular  elements of the ring $R$ (i.e. $\CC_R$ is the  set of non-zero-divisors of the ring $R$);

\item   $Q_{l,cl}(R):= \CC_R^{-1}R$  is the  {\em left quotient ring}   (the {\em  classical left   ring of fractions}) of the ring $R$ (if  it  exists);


\item $S_l(R)=S_{l,0}(R)$ is the largest element of the poset
$(\Den_l(R, 0), \subseteq )$ and $Q_l(R):=S_0^{-1}R$ is the
largest left quotient ring of $R$, \cite[Theorem 2.1.(2)]{larglquot}; 

\item For a   commutative ring $R$  and $\gp \in \Spec (R)$,   $R_\gp:=\Big(R\backslash \gp \Big)^{-1}R$ is the localization of $R$ at the prime ideal $\gp$; 

\item $\mS$ and $\mS_f$ are  the categories  of semisimple Artinian rings and finite dimensional semisimple algebras (over a fixed field), respectively, and  $\mathcal{S}$ is the category of semiprimary rings.


       
\end{itemize}

A ring $R$ is called a {\bf semiprimary ring} if its radical $\rad (R)$ is nilpotent ideal and the factor ring $R/\rad (R)$ is a semisimple Artinian ring. 
\begin{equation}\label{4-CLA}
\begin{split}
\{ {\rm Simple\; Artinian\; rings} \}&\subset \{ {\rm Semisimple\; Artinian\; rings} \}\subset \{ {\rm One-sided\; Artinian\; rings} \}\\
&\subset \{ {\rm Semiprimary\; rings} \}.
\end{split}
\end{equation}
The four categories of rings above are denoted
respectively by $\mathfrak{S}\subseteq \mathbb{S}\subseteq\mathbb{A}\subseteq  \mathcal{S}$. The category $\mS_f$  of finite dimensional semisimple algebras (over a fixed field) is a  subcategory of $\mS$.\\

{\bf Semiprime rings with finitely many minimal primes.} The ideal $$\gn_R:=\bigcap_{P\in \Spec (R)}P=\bigcap_{\gp\in \min(R)}\gp$$ is called the {\bf prime radical} of $R$. The ring $R$ is  a  semiprime ring iff $\gn_R=0$. An ideal $\ga$ of $R$ is called a {\bf semiprime ideal} if the factor ring $R/\ga$ is a semiprime ring. An ideal $\ga$ of $R$ is a semiprime ideall iff it is an intersection of prime ideals.

Prime and semiprime rings are large and important classes of rings both in commutative and non-commutative algebra. They also have geometric flavour. The algebra of regular functions on an affine algebraic variety is a semiprime algebra and the corresponding  ideal of definition of the variety is a semiprime ideal. If, in addition, the variery is irreducible that the algebra of regular functions is a prime algebra and the ideal of definition is a prime ideal.

\begin{itemize}

\item (Corollary \ref{a13Sep23}) {\em A semiprime ring with infinitely many minimal primes cannot be embedded into a semisimple Artinian ring.}

\item  (Corollary \ref{b13Sep23}) {\em A semiprime ring with infinitely many minimal primes cannot be embedded into a semiprimary ring. }

\end{itemize}

So, for a semiprime ring $R$ to be embeddable into a semisimle Artinian ring or into a semiprimary ring the condition $|\min (R)|<\infty$ is a necessary condition (but not sufficient, in general). Therefore, the condition $|\min (R)|<\infty$ is present in all results on embeddability.   The set of minimal primes of a ring is a very  important set as far the  spectrum of a ring is concerned as every prime contains a minimal prime. So, knowing the minimal primes is the first (important and difficult) step in describing the spectrum. 
In the algebraic geometry, the minimal primes of the algebra of regular functions on an algebraic variety determine/correspond to the   irreducible components of the variety.

In \cite{MinPrimeLocRings},
  several descriptions of the set of minimal prime ideals of localizations of rings $R$ with  $\min (R)|<\infty$ and satisfying some  natural assumptions  are obtained. In particular, the following case are considered: a  localization of a semiprime ring with finite set of minimal primes; a  localization of a prime rich ring  where the localization respects the ideal structure of primes and primeness of certain minimal primes; a localization of a ring at a left denominator set generated by normal elements, and others. As an application, for a semiprime ring with finitely many minimal primes, a description of the minimal primes of it's largest left/right quotient ring is obtained. We use several results of \cite{MinPrimeLocRings}, see Theorem \ref{A10Sep23}.\\

{\bf The category $\CM_{R,\mS}$ and the family of left $R$-collections $\CV_l(R)$.} 
At the beginning of Section \ref{EMBEDSPRIME}, for a ring $R$, we introduce and study the monomorphism category $\CM_{R, \mS}$. The objects of $\CM_{R, \mS}$ are all the monomorphisms from $R$ into semisimple Artinian rings and morphisms are defined by commutativity of obvious triangles.  A natural binary relation of  `smaller monomorphism' is introduced. Also there is an additional/refined definition of smallness (associated with sizes of matrices of semisimple Artinian rings). 

A  monomorphism $R\stackrel{f}{\ra}\prod_{i=1}^sA_i\in \CM_{R,\mS}$ is called {\bf redundant} if one of the direct multiplies can be omitted and {\bf irredundant} otherwise. Let $\Irr \,\CM_{R, \mS}$ be the  family of irredundant $R$-monomorphisms. 

One of the main concepts of the paper is `a left $R$-collection.' Intuitively, left $R$-collection is  a concept and a very efficient tool for producing  embeddings of a  semiprime ring $R$ with $|\min (R)|<\infty$ into semisimple Artinian rings. Furthermore, for a given embedding,   we can construct a smaller $R$-embedding by using left $R$-collections and we can construct  explicit proper epimorphisms/monomorphisms between the targets to demonstrate the  `smallness', see Theorem \ref{B9Sep23}.(2).

Suppose that  $R$ is  a semiprime ring with $|\min (R)|<\infty$ and 
$
A=\prod_{\gp \in \min (R)}A(\gp)
$
 is a semisimple Artinian ring where $A(\gp)$ are simple Artinian rings. 
 A set
 $$
V=\{ V(\gp)\, | \, \gp \in \min (R)\}
$$
 is called a {\bf left $R$-collection} if for every $\gp \in \min (R)$, 
 \begin{itemize}
 \item $V(\gp)$ is an simple $(R,A(\gp ))$-bimodule,
 \item  $l_\gp :=l(V(\gp)_{A(\gp )})<\infty$, and 
 \item $\gp = \ann_R(V(\gp))$.
 \end{itemize}
  Let $\CV_l(R)$ be a family of all left $R$-collections (for all possible rings $A$). For the  left $R$-collection $V$, there is an explicit monomorphism
$$\phi_V: R\ra A_V:=\End(V_A)=\prod_{\gp\in \min (R)}A_V(\gp)\in \CM_{R,\mS},\;\; {\rm  where}\;\; A_V(\gp):=\End(V(\gp) _{A(\gp)}) $$
are simple Artinian rings. The monomorphism $\phi_V$  is called the {\bf monomorphism/embedding of the collection} $V$ and all such monomorphisms are called {\bf collection $R$-monomorphisms}. Let $\CM_{R, \mS}^c$ be the class of  collection $R$-monomorphisms. 
 Proposition \ref{A14Sep23}  and Lemma \ref{b14Sep23} contain properties of left collections $V\in \CV_l(R)$, the rings  $A=\prod_{\gp \in \min (R)}A(\gp)$ and  $A_V$,    and the monomorphisms $\phi_V$, and other objects that are connected with left $R$-collections.

Suppose that $R$ is a semiprime ring with  $|\min (R)|<\infty$.  We say that an $R$-monomorphism $f: R\ra A=\prod_{\gp \in \min (R)}A(\gp)\in  \CM_{R, \mS}$, where $A(\gp)$ are simple Artinnian rings,  is  {\bf natural}  if $$\lann_A(\gp)=\rann_A(\gp)=A(\gp)\;\; {\rm for\; all}\;\;\gp \in \min (R).$$ Let $\CN_{R, \mS}$ be the class of all natural $R$-monomorphisms.

 A natural  embedding $R\stackrel{f}{\ra}A=\prod_{\gp \in \min (R)}A(\gp)\in \CN_{R,\mS}$ is called an {\bf elementary embedding} if
${}_R{A(\gp)}_{A(\gp)}$ is a simple bimodule for all $\gp \in \min (R)$. The class of  elementary embedding is denoted by $\CN_{R, \mS}^e$.
Every elementary embedding $R\stackrel{f}{\ra}A=\prod_{\gp \in \min (R)}A(\gp)\in \CN_{R,\mS}^e$ is a collection embedding, 
$f=\phi_{V_f}$,
for the left $R$-collection $V_f:=\{ V_f(\gp ):= A(\gp)\, | \, \gp \in  \min (R)\}\in \CV_l(R)$ since $\End (A(\gp)_{A(\gp)})=A(\gp)$ for all $\gp \in\min (R)$. The relations between the concepts are given below
$$\CN_{R, \mS}^e\subseteq \CM_{R, \mS}^c\subseteq \CN_{R, \mS}\subseteq \Irr \CM_{R, \mS}$$
Proposition \ref{a16Sep23} describes properties of natural monomorphisms.

Theorem \ref{B9Sep23} is one of the key results of the paper. It captures the essence of the structure of embeddings $\CM_{R, \mS}$. It reveals that left $R$-collections play  the central role in understanding of the structure of embeddings $\CM_{R, \mS}$. It shows that for each  embedding in $f\in \CM_{R, \mS}$ there is a smaller embedding which is  a collection embedding $\phi_V$, and it produces {\em explicit} morphisms in $\Mor (\CM_{R, \mS})$ that proves the smallness.  So, the proof is `constructive'. 

Several corollaries of Theorem \ref{B9Sep23} are obtained (Corollary \ref{a14Sep23}, Corollary \ref{a9Sep23} and Corollary \ref{b16Sep23}).

Suppose that $R$ is a semiprime ring with  $|\min (R)|<\infty$.  We say that natural $R$-monomorphisms $f\ra A=\prod_{\gp \in \min (R)}A(\gp)$  and $f'\ra A'=\prod_{\gp \in \min (R)}A'(\gp)$ are  {\bf $M$-equivalent} and write $f\stackrel{M}{\sim}f'$   if the rings  $A(\gp)$ and $A'(\gp)$ are Morita equivalent  for all $\gp \in \min (R)$. Let $[f]_M^\sim$ be the $M$-equivalence class of the element $f$ in $\CN_{R, \mS}$.

Theorem \ref{15Sep23} describes minimal embeddings in each $M$-equivalence class of embeddings. As a corollary, we have the equalities (Corollary \ref{e15Sep23})
 $$\min \CM_{R,\mS}=\min \CN_{R,\mS}=\min \CN_{R,\mS}^e.$$

The category $\mS_f$  of finite dimensional semisimple algebras (over the fixed ground field) is  closed under the Morita equivalence. 
 For category  $\mS_f$  all the previous results hold and they can be strengthened, see Corollary 
\ref{b15Sep23} and Corollary \ref{c15Sep23}.\\

{\bf  Criterion for a semiprime ring being embeddable into  semisimple Artinian ring.} 
At the beginning of Section \ref{LLQRINGS}, we collect results on localizations and the largest left quotient ring that are used in proofs of the next two sections. For a semiprime ring $R$ with finitely many minimal primes, Theorem \ref{A21Sep23} (resp.,  Corollary \ref{aA21Sep23}) shows that the ring $R$ and all its left localizations at regular left  (resp., right ) denominator sets  have the same elementary embeddings. So, in order  to describe all elementary embeddings  for these rings it suffices to do it for any of them. As localizations make the structure of rings `simpler' it suffices to describe the elementary embeddings for the largest left (resp., right) quotient ring of the ring. 

Theorem \ref{B21Sep23} describes the minimal embeddings for all semiprime left Goldie rings and their localizations.

Theorem \ref{21Sep23} is a criterion for a semiprime ring $R$ being embeddable into  semisimple Artinian   rings. 

\begin{theorem}\label{21Sep23}
Let $R$ be a semiprime ring. The following statements are equivalent:

\begin{enumerate}
\item $\CM_{R, \mS}\neq \emptyset$.
\item  $|\min (R)|<\infty$ and the ring $Q_l(R)$ is embeddable into a semisimple Artinian ring.
\item  $|\min (R)|<\infty$ and the ring $Q_l(R/\gp)$ is embeddable into a simple Artinian ring for every $\gp \in\min (R)$.
\item  $|\min (R)|<\infty$ and the ring $Q_r(R)$ is embeddable into a semisimple Artinian ring.
\item $|\min (R)|<\infty$ and the ring $Q_r(R/\gp)$ is embeddable into a simple Artinian ring for every $\gp \in \min (R)$.
\end{enumerate}
\end{theorem}

The proof of Theorem \ref{21Sep23} is given in Section \ref{LLQRINGS}.\\

{\bf New criterion for a ring to have a semisimple Artinian quotient ring.}
A criterion for a ring to have a simple Artinian left quotient ring was given by Goldie (\cite{Goldie-PLMS-1958},  1958) and Lesieur and Croisot (\cite{Lesieur-Croisot-1959}, 1959). Goldie's Theorem (\cite{Goldie-PLMS-1960}, 1960)  is a criterion for a ring to have a semisimple Artinian left quotient ring.  Goldie's Theorem states that {\em a ring has a  semisimple left quotient ring iff it is a semiprime left Goldie ring}. A ring is called a {\em left Goldie ring} if it satisfies the ascending chain condition on left annihilators and does not contain infinite direct sums of nonzero left ideals. In \cite{NewCrit-S-Simpl-lQuot}, four new criteria are given that are based on a completely different approach and new ideas. The first one is based on the fact that {\em for an  arbitrary ring $R$ the set $\CM := \maxDen_l(R)$ of maximal left denominator sets of $R$ is a non-empty set},  \cite[Lemma 3.7.(2)]{larglquot}. The First Criterion is given via properties of the  set $\CM$.\\

{\bf Theorem (The First Criterion, \cite[Theorem 3.1]{NewCrit-S-Simpl-lQuot}}. {\em A ring $R$ has a semisimple left quotient ring $Q_{l,cl}(R)$ iff $\CM $ is a finite set, $\bigcap_{S\in \CM } \ass (S) =0$ and,  for each $S\in \CM$, the ring  $S^{-1}R$ is a simple left Artinian ring. In this case, $Q_{l,cl}(R)\simeq \prod_{S\in \CM} S^{-1}R$.  }\\

 The Second Criterion, \cite[Theorem 4.1]{NewCrit-S-Simpl-lQuot},  is given via the minimal primes of $R$ and goes further than the First one 
  in the sense that it describes  explicitly the maximal left denominator sets $S$ via the minimal primes of $R$. The Third Criterion, \cite[Theorem 5.1]{NewCrit-S-Simpl-lQuot},  is close to Goldie's Criterion but it is easier to check in applications (basically, it reduces Goldie's Theorem to the prime case). The Fourth Criterion, \cite[Theorem 6.2]{NewCrit-S-Simpl-lQuot},  is given via certain left denominator sets. As far as applications are concerned,  it is a very efficient tool in showing that the left quotient ring is a semisimple Artinian  ring provided that the original ring admits sufficiently (finitely)  many `nice' localizations.

In Section \ref{NEWCRIT-THMG},  a new  criterion for the left quotient ring of a ring to be a semisimple Artinian ring is given (Theorem \ref{VVB-23Sep2023}). The criterion 
answers the following question: \\

$\bullet$ {\em Find necessary and and sufficient conditions on the centre of a semiprime ring that guarantee that the left quotient ring of a ring is a semisimple Artinian ring.}

\begin{theorem}\label{VVB-23Sep2023}
Let $R$ be a ring and $Z(R)$ be its centre. Then the  following statements are equivalent:
\begin{enumerate}

\item The ring $Q_{l,cl}(R)$ is a semisimple Artinian ring.

\item The ring $R$ is a semiprime ring, $\CC_{Z(R)}\subseteq \CC_R$,  $|\min (Z(R))|<\infty$ and for each $\gq\in   \min (Z(R))$, the ring $R_\gq$ is a left Goldie ring.

\end{enumerate}
If one of the equivalent conditions holds then $Q_{l,cl}(R)\simeq\prod_{\gq\in \min   (Z(R))}Q_{l,cl}(R_{\gq})$, $Q_{l,cl}(R_{\gq})\simeq Q_{l,cl}(R)_{\gq}$ for all $\gq\in \min (Z(R))$, and the map $\min (R)\ra \min (Z(R))$, $\gp\mapsto \gp^{\rm r}:=\gp\cap Z(R)$ is a surjection.
\end{theorem}

For a ring $R$  such that its centre $Z(R)=\prod_{i=1}^nK_i$ is a finite direct product of fields $K_i$, Corollary \ref{B23Sep23} is a criterion for the ring $Q_{l,cl}(R)$ being  a semisimple Artinian ring.

 In general, for a ring $R$, the classical left quotient ring $Q_{l,cl}:=\CC_R^{-1}R$ does not exists (e.g., the algebras of polynomial  integro-differential operators)  but for each ring $R$ there is the largest left Ore set  $S_l(R)$ of regular elements of $R$ and the ring $Q_l(R):=S_l(R)^{-1}R$ is called the the largest left quotient ring, \cite{Bav-intdifline, larglquot}.


\section{Embeddings of semiprime  rings  into  semisimple Artinian or  semiprimary rings.} \label{EMBEDSPRIME} 

In this section, three types of embeddings of semiprime rings into semisimple Artinian  rings, the elementary, natural and collection embeddings,  are  studied. The main tool is the concept of left $R$-collections. \\

{\bf The monomorphism/embedding category $\CM_{R, \CA}$.} 

\begin{definition} Let $R$ be a ring and $\CA$ be a class of rings. The {\bf monomorphism/embedding category} $\CM_{R, \CA}$ of the ring $R$ into $\CA$  is a category where the  objects $\Ob (\CM_{R, \CA})$ of  $\CM_{R, \CA}$ contains  all the   monomorphisms $R\stackrel{f}{\ra}A$ where $A\in \CA$ and the set of morphisms $\CM_{R, \CA} (f,g)$ between 
$R\stackrel{f}{\ra}A$ and $R\stackrel{g}{\ra}B$
are all possible homomorphisms $\alpha : A\ra B$ such that $g=\alpha f$. For a monomorphism $R\stackrel{f}{\ra}A$,  the ring $A$ is called the {\bf target} of $f$ and is denoted by $t(f)$.
\end{definition}

For each $R\stackrel{f}{\ra}A\in \CM_{R, \CA}$,
the identity map $\id_A:A\ra A$, $a\mapsto a$ is the identity of the endomorphism ring $\CM_{R, \CA}(f,f)\simeq \End_R(A)$. Two monomorhisms/objects  $R\stackrel{f}{\ra}A$ and $R\stackrel{g}{\ra}B$  of $\CM_{R, \CA}$ are isomorphic iff the set $\CM_{R, \CA} (f,g)$ contains an isomorphism iff the set $\CM_{R, \CA} (g,f)$ contains an isomorphism. For a monomorphism $R\stackrel{f}{\ra}A$,  its isomorphism class is denoted by $[R\stackrel{f}{\ra}A]$ or $[f]$.

A subset $S'$ of a set $S$ is called a {\bf proper subset} of $S'$ if $S'\neq \emptyset , S$. 
Let $A$ and $B$ be rings. A ring epimorphism $\alpha :A\ra B$ which is not an isomorphism is called a {\bf proper epimorphism}. Similarly, a ring monmorphism $\alpha :A\ra B$ which is not an isomorphism is called a {\bf proper monomorphism}.\\

{\bf The partial order $\leq $ on $\Ob (\CM_R,\CA)$.} Let us consider a binary relation $\leq$ on the set of object $\Ob (\CM_{R,\CA})$: For two objects $f,g\in \Ob (\CM_{R,\CA})$, we write $f\leq g$ if either $f=\alpha g$ for some  $\alpha \in \CM_{R, \CA}(g,f)$ which is an epimorphism from $t(g)$ to $t(f)$ or $g=\beta f$  for some  $\beta \in \CM_{R, \CA}(f,g)$ which is a  monomorphism from $t(f)$ to $t(g)$. 
  We denote by `$\leq$' the minimal partial order  such that the binary relation $\leq$ generates. If $f\leq g$, we say that $f$ is {\em smaller} (or equal) to $g$. If $f\leq g$ but $g\not\leq f$, we say that $f$ is {\em strictly smaller}  that $g$ and write $f<g$.

Intuition behind this partial order is to
 compare two monomorphisms using their images. 
 Algebraically, there are essentially two different ways of how to construct a  `smaller' monomorphism   in one step from a given one, say $R\stackrel{f}{\ra}A$: either to take a proper subring in $A$ that belongs to the class $\CA$ and  contains the image $\im (f)$ of $f$ or to take the proper epimorphism of $A$ such that the kernel of it has zero intersection with $R$. So, for a given $f$ all smaller  $g$'s, i.e.  $g\leq f$,  are obtained from $f$ in several consecutive  steps as have just described.

\begin{definition}
An object  $f\in \Ob (\CM_{R,\CA})$ is called {\bf tiny} if for each object $g\in \Ob (\CM_{R,\CA})$ the sets of morphisms  $\Mor (f,g)$ and $\Mor (g,f)$ consists entirely of isomorphisms  provided that they are non-empty. 
\end{definition}

It is obvious that the class $\min \CM_{R,\CA}:=\min \Ob (\CM_{R,\CA})$ of minimal objects of the partially ordered set  $(\Ob (\CM_{R,\CA}), \leq)$ contains all the tiny objects.  In general, the set $\min \CM_{R,\CA}$ is an empty set unless some `artinian/finiteness conditions' hold for the  category $\CA$. 

\begin{example} Let $R=K[x]$ be a polynomial algebra in a variable $x$ over a field $K$ and $\CA = \{ R\}$. Then $\min \CM_{R, \CA}=\emptyset$.
\end{example}

\begin{example} Let $R$ be a semiprime left Goldie ring. By Goldie's Theorem, $Q_{l,cl}(R)$ is a semisimple Artinian ring. Therefore, $\CM_{R, \mS}\neq \emptyset$. Furthermore,
$R\stackrel{\s_R}{\ra}Q_{l,cl}(R)\in \min \CM_{R,  \mS}$ where $\s_R(r)=\frac{r}{1}$ for all $r\in R$, and so $\min \CM_{R, \mS}\neq \emptyset$ (Theorem \ref{B21Sep23}).
\end{example}

{\bf The algebras of polynomial  integro-differential operators $\mI_n$.}
 Let $K$ be a field of characteristic zero and 
 $$\mI_n=K\bigg\langle x_1, \ldots , x_n, \frac{\der}{\der x_1}, \ldots ,\frac{\der}{\der x_n}, \int_1 \ldots , \int_n  \bigg\rangle
 $$ be the {\bf algebra of polynomial  integro-differential operators}, see \cite{Bav-intdif-mIn, Bav-intdifline} for details where the algebras $\mI_n$  are introduced and studied. The algebra $\mI_n$ contains the {\bf Weyl algebra} $$A_n:=K\bigg\langle x_1, \ldots , x_n, \frac{\der}{\der x_1}, \ldots ,\frac{\der}{\der x_n}\bigg\rangle
 ,$$ i.e. the algebra of polynomial differential operators. The Weyl algebra $A_n$ is a  Noetherian domain. Hence, by Goldie's Theorem its classical left quotient ring $Q_{l,cl}(A_n)$ coincides with its classical right quotient ring $Q_{r,cl}(A_n)$ which are division algebras. Hence, the Weyl algebra $A_n\subseteq Q_{l,cl}(A_n)$ is embeddable into a semisimple Artinian ring. To the contrary, for the algebras $\mI_n$  neither left nor right classical quotient ring exists. Furthermore, the algebras $\mI_n$ are not embeddable into semisimple Artinian rings, see Lemma \ref{a17Sep23}.(2).

\begin{lemma}\label{a17Sep23}

\begin{enumerate}
\item If a ring $R$ contains an infinite set of orthogonal nonzero idempotents then $\CM_{R, \mS}=\emptyset$.
\item For all $n\geq 1$, $\CM_{\mI_n, \mS}=\emptyset$.
\item If a ring $R$ contains elements $x$ and $y$  such that $yx=1$ but $xy\neq 1$ then $\CM_{R, \mS}=\emptyset$.
\item If $\CM_{R, \mS}=\emptyset$ then $\CM_{R', \mS}=\emptyset$ for all rings $R'$ that contain the ring $R$.
\end{enumerate}
\end{lemma}

\begin{proof} 1. Each  semisimple Artinian ring  does not contain an infinite set of orthogonal nonzero idempotents, and statement 1 follows.

2. The algebras $\mI_n$ contain an infinite set of orthogonal idempotents  \cite[Eq.(1)]{Bav-intdifline}, and so statement 2 follows from statement 1.

3. Then set $\{e_i:=x^iy^i-x^{i+1}y^{i+1}\, |\, i\in \N\}$ is an infinite set of orthogonal nonzero idempotents. Direct calculations show that the idempotents are orthogonal. They are also  nonzero: If $0=e_i$ for some $i\geq 0$ then
 $$0=y^i0x^i=y^ie_ix^i=1-xy, \;\; {\rm i.e.}\;\;xy=1,
 $$
  a contradiction. Now, statement 3 follows from statement 1. 

4. Statement 4 is obvious.
\end{proof}

{\bf The category $\CM_{R,\mS}$ and the family of left $R$-collections $\CV_l(R)$.}  
 Let $A$ be a semisimple Artinian ring. The ring $$A=\prod_{i=1}^s A_i$$ is a direct product of simple Artinian rings $A_i=M_{n_i}(D_i)$ where $n_i\geq 1$ and $D_i$ are division rings. Let $1=e_1+\cdots +e_s$ be the corresponding  sum of 
 central idempotents of $A$.
  Let  $U_i$ be a unique (up to isomorphism) simple  left  $A_i$-module. Namely,  ${}_{D_i}U_i=D_i^{n_i}$ with matrix multiplication on the left. Each nonzero $A$-module of finite length $$N=\bigoplus_{i=1}^s N_i$$ is a unique direct sum where $N_i=e_iA\simeq U_i^{l_i}$ for some $l_i\in \N$, and $l_A(N)=\sum_{i=1}^sl_i$ is the length of the $A$-module $N$. The unique $A$-submodule $N_i$ is called  the $i$'th {\bf isotypic component} of $N$. The $A$-module $N_i$ is the sum of all simple left $A$-submodules of $N$ that are isomorphic to $U_i$  or zero if there are none. The $A$-modules $U_1, \ldots , U_s$  are all the simple $A$-modules up to isomorphism.
  For each $i=1, \ldots , s$, 
  $$\End_A(U_i)\simeq D_i.$$ We write endomorphism on the opposite side of the scalars but for  practical reason individual endomorphism are often  written on the left.

Similarly, let  $V_i$ be a unique (up to isomorphism) simple  right  $A_i$-module. Namely,  ${V_i}_{D_i}=D_i^{n_i}$ with matrix multiplication on the right.
Each nonzero right $A$-module of finite length $$M=\bigoplus_{i=1}^sM_i$$ is a unique direct sum where $M_i=Ae_i\simeq V_i^{l_i}$ for some $l_i\in \N$, and $l (M_A)=\sum_{i=1}^sl_i$ is the length of the right $A$-module $M$. The unique right $A$-submodule $M_i$ is called  the $i$'th {\bf isotypic component} of $M$. The right $A$-module $M_i$ is the sum of all simple right $A$-submodules of $M$ that are isomorphic to $V_i$ or zero if there are none. 
  For each $i=1, \ldots , s$, 
  $$\End_A(V_i)\simeq D_i.$$ 
The $A$-modules $V_1, \ldots , V_s$  are all the simple right  $A$-modules up to isomorphism.

\begin{definition}
A  monomorphism $R\stackrel{f}{\ra}\prod_{i=1}^sA_i\in \CM_{R,\mS}$ is called {\bf redundant} if one of the direct multiplies can be omitted and {\bf irredundant} otherwise. Let $\Irr \,\CM_{R, \mS}$ be the  family of irredundant $R$-monomorphisms. For an irredundant monomorphism $R\stackrel{f}{\ra}\prod_{i=1}^sA_i\in \CM_{R,\mS}$, the natural number $s(f):=s$ is called the {\bf size} of $f$. 
\end{definition}

We will see that if $R$ is a semiprime ring with $|\min (R)|<\infty$ then $s(f)\leq |\min (R)|$ for all $R\stackrel{f}{\ra}\prod_{i=1}^sA_i\in \Irr \,\CM_{R,\mS}$ (Theorem \ref{B9Sep23}.(2a)).

\begin{lemma}\label{a12Sep23}
Suppose that a monomorphism $R\stackrel{f}{\ra}\prod_{i=1}^nA_i\in \CM_{R,\mS}$ is  redundant then there is a proper subset $I$ of $\{ 1, \ldots , n\}$ such that the  monomorphism $R\stackrel{f_I}{\ra}\prod_{i\in I}A_i\in \CM_{R,\mS}$ is  irredundant where $f_I:=\pr_I f$ and $\pr_I:\prod_{i=1}^nA_i\ra \prod_{i\in I}A_i$ is the projection epimorphism, and $f_I\leq f$.
\end{lemma}
\begin{proof} Straightforward.
\end{proof}

{\bf Left $(R,A)$-collections and $R$-collection monomorphisms.} Given rings $R$  and $A$, a $(R,A)$-bimodule $M= {}_RM_A$ is called a simple $(R,A)$-bimodule if $\{ 0\}$ and $M$ are the only $(R,A)$-subbimodules of $M$. For a left  (resp., right) $R$-module $N$, $l (N)=l_R(N)=l({}_RN)$ (resp., $l (N)=l(N_R)$) is its length and $\ann (N)=\ann ({}_RN)$ (resp., $\ann (N)=\ann (N_R)$) is its annihilator. 
\begin{definition}\label{collect}
Suppose that  $R$ is  a semiprime ring with $|\min (R)|<\infty$ and 

\begin{equation}\label{A=pA(gp)}
A=\prod_{\gp \in \min (R)}A(\gp)
\end{equation}
 is a semisimple Artinian ring where $A(\gp)$ are simple Artinian rings and 
\begin{equation}\label{1=seA(gp)}
1=\sum_{\gp \in \min (R)}e_{A(\gp)}
\end{equation} 
be the corresponding sum of central orthogonal idempotents $e_{A(\gp)}$ of $A$. 
 A set
\begin{equation}\label{V=setV(gp)}
V=\{ V(\gp)\, | \, \gp \in \min (R)\}
\end{equation}
 is called a {\bf left $R$-collection} if for every $\gp \in \min (R)$, 
 \begin{itemize}
 \item $V(\gp)$ is an simple $(R,A(\gp ))$-bimodule,
 \item  $l_\gp :=l(V(\gp)_{A(\gp )})<\infty$, and 
 \item $\gp = \ann_R(V(\gp))$.
 \end{itemize}
  Let $\CV_l(R)$ be a family of all left $R$-collections (for all possible rings $A$). We drop the adjective `left' if this does not lead to confusion (`left' stands for `left $R$-module').
\end{definition}
If we want to stress the role of the ring $A$ in the $R$-collection $V$ then we say that $V$ is an $(R,A)$-collection.

Intuitively, left $R$-collection is  a concept and a very efficient tool for producing  embeddings of a  semiprime ring $R$ with $|\min (R)|<\infty$ into semisimple Artinian rings. Furthermore, for a given embedding,   we can construct a smaller $R$-embedding by using left $R$-collections and we can construct  explicit proper epimorphisms/monomorphisms between the targets to demonstrate the  `smallness'.

Each left collection $V=\{ V(\gp)\, | \, \gp \in \min (R)\}$ can be identified with the semisimple $(R,A)$-bimodule 
\begin{equation}\label{VRcol-5}
V=\bigoplus_{\gp \in \min (R)}V(\gp)
\end{equation}
where $V(\gp)=Ve_{A(\gp)}$ are simple isotypic components of ${}_RV_A$ such that $l(V_A)<\infty$ and $\gp = \ann_R(V(\gp))$ for all $\gp \in \min (R)$, and vice versa. In more detail,  if $V$ is such a $(R,A)$-bimodule then the set of its isotypic components
\begin{equation}\label{VRcol}
 \{ V(\gp):=Ve_{A(\gp)}\, | \, \gp \in \min (R)\}
\end{equation}
is a left $R$-collection.
\begin{definition}
The $(R,A)$-bimodule $V$ is called the {\bf $(R,A)$-bimodule  of the collection} $V$. 
\end{definition}
  For all $\gp, \gq\in \min (R)$ such that $\gp \neq \gq$, 
$$\Hom (V(\gp)_A,V(\gq)_A)=\{ 0\}$$
as for every $f\in \Hom (V(\gp)_A,V(\gq)_A)$,
$f(V(\gp))=f(V(\gp)e_{A(\gp)})e_{A(\gq)}=f(V(\gp))e_{A(\gp)}e_{A(\gq)}=0.$
It follows that 
\begin{equation}\label{VRcol-6}
\End (V_A)\simeq \prod_{\gp \in \min (R)}\End(V(\gp) _{A(\gp)})
\end{equation}
In order to avoid opposite rings, we write endomorphisms on the {\em opposite} site of the scalars. Though some individual endomorphisms we write on the left (if this does not lead to confusion and simplifies  notation). 

Let $V(A(\gp))$ be a unique (up to isomorphism) simple right $A(\gp)$-module and $U(A(\gp))$ be a unique (up to isomorphism) simple right $A(\gp)$-module. Then the endomorphism ring 
\begin{equation}\label{VRcol-7}
D(\gp):=\End (V(A(\gp))_{A(\gp)})\simeq \End ({}_{A(\gp)}U(A(\gp)))
\end{equation}
 is a division ring and so its centre $Z(D(\gp))$ is a field. Then 
\begin{equation}\label{VRcol-8}
A(\gp))\simeq M_{n_\gp}(D(\gp))\;\; {\rm and}\;\; 
V(A(\gp))_{A(\gp)}\simeq \Big( D(\gp)^{n_\gp}\Big)_{ M_{n_\gp}(D(\gp))}\;\; {\rm for\; some}\;\; n_\gp\geq 1.
\end{equation}
For each  $(R,A)$-collection $V\in \CV_l(R)$, let us consider the ring monomorphism 
\begin{equation}\label{VRcol-1}
\phi_V: R\ra \End (V_A)\simeq \prod_{\gp \in \min (R)}\End(V(\gp) _{A(\gp)})\simeq \prod_{\gp \in \min (R)}M_{l_\gp}(D(\gp))
\end{equation}
where for all elements $r\in R$ and $(v(\gp))_{\gp \in \min (R)}\in V=\bigoplus_{\gp \in \min (R)}V(\gp)$,
\begin{equation}\label{VRcol-2}
\phi_V(r)(v(\gp))_{\gp \in \min (R)} : =(rv(\gp))_{\gp \in \min (R)}.
\end{equation}
In more detail, since the ring $R$ is a semiprime ring and $\ann_R(V(\gp))=\gp$ for all $\gp \in \min (R)$, 
$$\ker (\phi_V)=\bigcap_{\gp \in \min (R)}\ann_R(V(\gp))=\bigcap_{\gp \in \min(R)}\gp=0,$$ 
i.e. the ring homomorphism  $\phi_V$ is a monomorphism. 

\begin{definition}
The above monomorphism 
$$\phi_V: R\ra A_V:=\End(V_A)=\prod_{\gp\in \min (R)}A_V(\gp)\in \CM_{R,\mS},\;\; {\rm  where}\;\; A_V(\gp):=\End(V(\gp) _{A(\gp)}), $$ is called the {\bf monomorphism/embedding of the collection} $V$ and all such monomorphisms are called {\bf collection $R$-monomorphisms}. Let $\CM_{R, \mS}^c$ be the class of  collection $R$-monomorphisms. 
\end{definition}
For a $(R,A)$-collection $V\in   \CV_l(R)$, the endomorphism ring
\begin{equation}\label{VRcol-3}
\End ({}_RV_A)\simeq \prod_{\gp \in \min (R)}\End({}_RV(\gp)_{A(\gp)})
\end{equation}
is a direct product of {\em division rings} $\End({}_RV(\gp) _{A(\gp)})$ (since ${}_RV(\gp) _{A(\gp)}$ is a simple bimodule),   
\begin{equation}\label{VRcol-4}
\End ({}_RV_A)\subseteq \End (V_A)\;\; {\rm and}\;\; \End({}_RV(\gp)_{A(\gp)})\subseteq \End({}V(\gp)_{A(\gp)})=A_V(\gp)\;\; {\rm for\; all}\;\; \gp \in \min (R).
\end{equation}

Proposition \ref{A14Sep23}  and Lemma \ref{b14Sep23} contain some useful properties of $V\in \CV_l(R)$, $A=\prod_{\gp \in \min (R)}A(\gp)$, $A_V$   and $\phi_V$.

\begin{proposition}\label{A14Sep23}
Let $V\in \CV_l(R)$, $A=\prod_{\gp \in \min (R)}A(\gp)$,  and $\phi_V: R\ra A_V:=\prod_{\gp\in \min (R)}A_V(\gp)$ be as above. Then 
\begin{enumerate}

\item For each $\gp \in \min (R)$,  the bimodule ${}_{A_V(\gp)}V(\gp)_{A(\gp)}$ is  simple.

\item {\sc (The Double Centralizer Theorem)}
$\End ({}_{A_V(\gp)}V(\gp))\simeq A(\gp)$.

\item For each $\gp \in \min (R)$, $$Z(\End (V(A(\gp))_{A(\gp)}))=Z(D(\gp))=\End({}_{A_V(\gp)}V(\gp)_{A(\gp)})\subseteq 
\End({}_RV(\gp)_{A(\gp)})\subseteq \End(V(\gp)_{A(\gp)}).$$
 
\end{enumerate}
\end{proposition}

\begin{proof}  By (\ref{VRcol-7}), (\ref{VRcol-8})  and (\ref{VRcol-1}), $D(\gp)=\End(V(A(\gp)_{A(\gp)})$  is a division ring, $A(\gp)\simeq M_{n_\gp}(D(\gp))$,  $$V(\gp)_{A(\gp)}\simeq V(A(\gp))^{l_\gp}\;\; {\rm  and}\;\; A_V(\gp)=\End (V(\gp)_{A(\gp)})\simeq \End(V(A(\gp))^{l_\gp}_{A(\gp)})\simeq M_{l_\gp}(D(\gp)).$$ Let $ M_{l_\gp, n_\gp}(D(\gp))$ be the set of $l_\gp\times  n_\gp$-matrices over the division ring $D(\gp)$. Then  

\begin{equation}\label{VRcol-9}
{}_{A_V(\gp)}V(\gp)_{A(\gp)}\simeq {}_{M_{l_\gp}(D(\gp))}M_{l_\gp, n_\gp}(D(\gp))_{M_{n_\gp}(D(\gp))}.
\end{equation}

1. We have to show that every nonzero element $a=(a_{ij})\in 
M_{l_\gp, n_\gp}(D(\gp))$ is a bimodule generator for  ${}_{M_{l_\gp}(D(\gp))}M_{l_\gp, n_\gp}(D(\gp))_{M_{n_\gp}(D(\gp))}$. Since $a\neq 0$, $a_{ij}\neq 0$ for some $i$ and $j$. Then the element 
$a_{ij}^{-1}E_{ii}aE_{jj}=E_{ij}$ is a bimodule generator, and we are done. So,  the bimodule ${}_{A_V(\gp)}V(\gp)_{A(\gp)}$ is  simple, by (\ref{VRcol-9}).

2. The left $M_{l_\gp}(D(\gp))$-module $U(A_V(\gp))=D(\gp)^{l_\gp}$ is a unique (up to isomorphism) simple left $M_{l_\gp}(D(\gp))$-module, and $\End ({}_{M_{l_\gp}(D(\gp))}U(A_V(\gp)))\simeq D(\gp)$.  Now, by (\ref{VRcol-9}),
 $$\End ({}_{A_V(\gp)}V(\gp))\simeq \End( {}_{M_{l_\gp}(D(\gp))}M_{l_\gp, n_\gp}(D(\gp)))\simeq 
  \End ({}_{M_{l_\gp}(D(\gp))}\Big(D(\gp)^{l_\gp}\Big)^{n_\gp})\simeq M_{n_\gp}(D(\gp)) \simeq A(\gp).$$
3. Clearly, 
$\End({}_{A_V(\gp)}V(\gp)_{A(\gp)})\subseteq 
\End({}_RV(\gp)_{A(\gp)})\subseteq \End(V(\gp)_{A(\gp)})$ and $Z(D(\gp))=Z(\End (V(A(\gp))_{A(\gp)}))
$, see  (\ref{VRcol-7}).  In view of  (\ref{VRcol-9}), it suffices  to show that
$$ \End( {}_{M_{l_\gp}(D(\gp))}M_{l_\gp, n_\gp}(D(\gp))_{M_{n_\gp}(D(\gp))})\simeq Z(D(\gp)).
$$
Let $f\in  \End( {}_{M_{l_\gp}(D(\gp))}M_{l_\gp, n_\gp}(D(\gp))_{M_{n_\gp}(D(\gp))})$. Then $f(E_{11})=zE_{11}$ for some element $z\in D(\gp)$. For all elements $d\in D(\gp)$, 
$$f(0)=f(dE_{11}-E_{11}d)=df(E_{11})-f(E_{11})d=(dz-zd)E_{11},$$
and so the $z\in Z(D(\gp)$. 
 For all $i$ and $j$, $$f(E_{ij})=f(E_{i1}E_{11}E_{1j})=E_{i1}f(E_{11})E_{1j}=E_{i1}zE_{11}E_{1j}=zE_{ij}.$$

Now, for all elements $a=\sum_{i,j}a_{ij}E_{ij}\in M_{l_\gp, n_\gp}(D(\gp))$,
$$f(a)=f\bigg(\sum_{i,j}a_{ij}E_{ij}\bigg)=\sum_{i,j}a_{ij}f(E_{ij})=\sum_{i,j}a_{ij}zE_{ij}
=z\sum_{i,j}a_{ij}E_{ij}=za,$$
 and the claim follows.  
\end{proof}

For a subset $S$ of a ring $A$, the sets $\lann_A(S):=\{ a\in A\, | \, aS=0\}$ and 
 $\rann_A(S):=\{ a\in A\, | \, Sa=0\}$ are called the {\bf left} and {\bf right annihilator} of $S$ in $A$, respectively. By the definition, the sets $\lann_A(S)$ and $\rann_A(S)$ are left and right ideals, respectively. 


\begin{lemma}\label{b14Sep23}
Let $V\in \CV_l(R)$, $A=\prod_{\gp \in \min (R)}A(\gp)$,  and $\phi_V: R\ra A_V:=\prod_{\gp\in \min (R)}A_V(\gp)$ be as above. Then 
\begin{enumerate}
\item $\CM_{R, \mS}^c\subseteq \Irr \,\CM_{R, \mS}$.
\item For each $\gp \in \min (R)$, the rings $A(\gp)$ and $A_V(\gp)$ are Morita equivalent.

\item $\gp V(\gq)=\begin{cases}
 V(\gq)& \text{if }\gp \neq  \gq,\\
\{ 0\}& \text{if }\gp= \gq.
\end{cases}$

\item For all $\gp \in \min (R)$, the factor algebra $R/\gp$ is a subalgebra of $A_V(\gp)$. In particular, $\gp A_V(\gp)= A_V(\gp)\gp =\{0\}$. 


\item  For all $\gp \in \min (R)$, $\lann_{A_V}(\gp)=\rann_{A_V}(\gp)=A_V(\gp)$.

\item The right $A$-module $V =\bigoplus_{\gp \in \min (R)}V(\gp)$ is a finitely generated projective generator. The rings $A$ and $A_V$ are Morita equivalent.
\end{enumerate}
\end{lemma}

\begin{proof} 1. For $V\in\CV_l(R)$, $\gp= \ann_R(V(\gp)$ for all $\gp \in \min (R)$. Therefore, $\phi_V\in \Irr \,\CM_{R, \mS}$. Hence,  $\CM_{R, \mS}^c\subseteq \Irr \,\CM_{R, \mS}$.

 2. By (\ref{VRcol-1}), $A_V(\gp)\simeq M_{l_\gp}(D(\gp))$ and $A(\gp)\simeq M_{n_\gp}(D(\gp))$ where $n_\gp :=l_{A(\gp)}(A(\gp))$, and statement 1 follows.

3. Since $V\in \CV_l(R)$, $\gp = \ann_R(V(\gp))$, and so $\gp V(\gp)=0$ and $\gp V(\gq)\neq 0$ for all $\gq \neq \gp$ (otherwise, $\gp\subseteq \gq$, a contradiction). Since $\gp V(\gq)\neq 0$ is a nonzero $(R, A(\gq))$-subbimodule of the simple 
$(R, A(\gq))$-bimodule $V(\gq)$, we must have $\gp V(\gq)=V(\gq)$.

4.  By the definition of the collection $V$, $\gp = \ann_R(V(\gp))$. Now, by (\ref{VRcol-1})  and  (\ref{VRcol-2}), the factor algebra $R/\gp$ is a subalgebra of $A_V(\gp)$. In particular, $\gp A_V(\gp)= A_V(\gp)\gp =\{0\}$.

5. In view of statement 4 and (\ref{VRcol-6}),   it suffices to show that for each $\gq\neq \gp$ and each nonzero element $\alpha \in \End( V(\gq)_{A}))$, $$\gp \alpha\neq 0\;\; {\rm  and}\;\;\alpha \gp \neq 0.$$
(i) $\gp \alpha\neq 0$: Suppose that $\gp \alpha =0$. Since $0\neq \alpha \in \End( V(\gq)_{A}))$, the image of the endomorphisms $\alpha$, $\im (\alpha)$,  is a nonzero right $A(\gq)$-submodule of $V(\gq)_{A(\gq)}$. By the definition of the $R$-collection, the bimodule ${}_RV(\gq)_{A(\gq)}$ is simple. Hence,
$R\im (\alpha)=V(\gq)$.
Now, $$\{0\}=\gp \alpha (V(\gq))=\gp\im (\alpha)
=\gp R \im (\alpha)=\gp V(\gq)=V(\gq)\neq \{ 0\},$$  by statement 3 (since $\gp\neq \gq$), a contradiction.

(ii) $ \alpha\gp\neq 0$: Suppose that $ \alpha\gp =0$. Then  
$$\{ 0\}=\alpha \gp(V(\gq)) =\alpha (\gp V(\gq))=\alpha (V(\gq))\neq \{0\},$$
 by statement 3 (since $\gp\neq \gq$), a contradiction.


6. By the definition of the collection $V$, the right $A$-module $V =\bigoplus_{\gp \in \min (R)}V(\gp)$ is a finitely generated projective generator. Hence, the rings $A$ and $A_V\simeq \End (V_A)$ are Morita equivalent. Alternatively, statement 6 follows from statement 2 as 
$A=\prod_{\gp \in \min (R)}A(\gp)$ and $A_V=\prod_{\gp \in \min (R)}A_V(\gp)$.
\end{proof}

{\bf Characterization of the set of minimal primes of a semiprime ring.}
\begin{definition} A finite set $Q$ of ideals of a ring  $R$ with zero intersection is called an {\bf irredundant  set} of ideals if $\bigcap_{\gq \in Q\backslash \{ \gq'\}}\gq\neq 0$ for all $\gq'\in Q$. 
\end{definition}

Lemma \ref{b10Sep23} is a useful characterization of  the set $\min (R)$ of minimal primes of a semirpime ring $R$  provided $|\min (R)|<\infty$. 

\begin{lemma}\label{b10Sep23}
Let $R$ be a semiprime ring with $|\min (R)|<\infty$. Then the  set $\min (R)$ is the only irredundant set that consist of prime ideals of $R$.
\end{lemma}
\begin{proof} Evidently, the  set $\min (R)$ is an  irredundant set that consist of prime ideals of $R$.

Suppose that $Q$ is an  irredundant set that consist of prime ideals of $R$. We have to show that $Q=\min (R)$. Since the set $\min (R)$ is an irredundant set of prime ideals, it suffices to show that $Q\supseteq\min (R)$. For each $\gp \in  \min (R)$,
$$ \bigcap_{\gq\in Q}\gq =\{0\}\subseteq \gp,$$
and so $\gq\subseteq \gp$ for some $\gq \in Q$. Hence $\gp=\gp$, by the minimality of $\gp$. Therefore, $Q\supseteq\min (R)$. 
\end{proof}

{\bf Natural embeddings and the $M$-equivalence relations.} 
\begin{definition}
Suppose that $R$ is a semiprime ring with  $|\min (R)|<\infty$.  We say that an $R$-monomorphism $f: R\ra A=\prod_{\gp \in \min (R)}A(\gp)\in  \CM_{R, \mS}$, where $A(\gp)$ are simple Artinnian rings,  is called  {\bf natural}  if $\lann_A(\gp)=\rann_A(\gp)=A(\gp)$
 for all $\gp \in \min (R)$. Let $\CN_{R, \mS}$ be the class of all natural $R$-monomorphisms. 
\end{definition}

\begin{proposition}\label{a16Sep23}
Suppose that $R$ is a semiprime ring with  $|\min (R)|<\infty$. Then 
\begin{enumerate}
\item $\CM_{R, \mS}^c\subseteq \CN_{R, \mS}\subseteq \Irr \,\CM_{R, \mS}$.
\item Suppose that  $f: R \ra A=\prod_{\gp \in \min (R)}A(\gp)\in  \CN_{R, \mS}$. Then  
\begin{enumerate}

\item  For all $\gp \in \min (R)$,  $\lann_A(A\gp A)=\rann_A(A\gp A)=A(\gp)$.

\item For all $\gp \in \min (R)$,  $A\gp A=A(1-e_{A(\gp)})=\prod_{\gq\in \min    (R)\backslash \{ \gp\} }A(\gq)$ where $1=\sum_{\gp \in \min (R)}e_{A(\gp)}$ is the sum of central orthogonal idempotents of $A$ that corresponds to the direct product $A=\prod_{\gp \in \min (R)}A(\gp)$.

\item For all $\gp \in \min (R)$, $A/A\gp A\simeq A(\gp)$.

\item For all $\gp \in \min (R)$, $R\cap A\gp A=\gp$ and $R/\gp = R/R\cap A\gp A\subseteq A/A\gp A\simeq A(\gp)$, i.e. the ring $R/\gp$ is a subring of $A(\gp)$.

\end{enumerate}

\item For all $f,f'\in  \CN_{R, \mS}$, $\CM_{R,\mS}(f,f')$ consists of ring monomorphisms.
\end{enumerate}
\end{proposition}

\begin{proof} 1. By Lemma \ref{b14Sep23}.(5), $\CM_{R, \mS}^c\subseteq \CN_{R, \mS}$.

Suppose that  $f: R\ra A=\prod_{\gp \in \min (R)}A(\gp)\in  \CN_{R, \mS}$.
  Then $\rann_A(\gp)=A(\gp)$ for all $\gp \in \min (R)$, and so $f\in \Irr \,\CM_{R, \mS}$
   since  for all $\gp\in \min (R)$,
   $\gp^c=\bigcap_{\gq\in \min (R)\backslash \{ \gp\} }\gq\neq 0$ (by Lemma \ref{b10Sep23})  and $\gp^c\subseteq \ker (f_\gp)$ where
   $$f_\gp: R\ra A\ra \prod_{\gq \in \min (R)\backslash \{ \gp\} }A(\gq)$$
 and the second map is the projection epimorphism. Therefore $ \CN_{R, \mS}\subseteq \Irr \,\CM_{R, \mS}$.

2(a). Since the right annihilator  $\rann_A(\gp)=A(\gp)$ is an ideal of $A$, we have that  $\rann_A(A\gp A)=A(\gp)$.  Since the ring $A$ is a semiprime ring,   $\lann_A(A\gp A)=\rann_A(A\gp A)=A(\gp)$, and the statement (a) follows.

2(b). Since the ring $A=\prod_{\gp \in \min (R)}A(\gp$ is a finite direct product of simple Artinian rings $A(\gp)$, the statement (b) follows from the statement (a).

2(c). The statement (c) follows from the statement (b): $A/A\gp A=A/A(1-e_{A(\gp)})\simeq A(\gp)$.

2(d). For each $\gp \in \min (R)$, the ideal $\gp':=R\cap A\gp A$ of the ring  $R$ contains the ideal $\gp$.
 Since $f: R \ra A=\prod_{\gp \in \min (R)}A(\gp)\in  \CM_{R, \mS}$,
 $$ \ga:=\bigcap_{\gp\in \min (R)}\gp'\subseteq \ker (f)=\{0\},$$ and so $\ga =\{ 0\}$.
 For each $\gp\in \min (R)$, the inclusion 
$\ga=\{0\}\subseteq \gp$, implies that $$\gq\subseteq \gq'\subseteq \gp\;\; {\rm  for\; some}\;\; \gq\in \min (R).$$ Since the minimal primes are incomparable, we must have $\gq=\gp$ and $\gp'\subseteq \gp$. The reverse inclusion also holds. So $\gp'=\gp$.

3. Let $f: R\ra A=\prod_{\gp \in \min (R)}A(\gp)\in  \CN_{R, \mS}$,  $f': R\ra A'=\prod_{\gp \in \min (R)}A'(\gp)\in  \CN_{R, \mS}$ and $\alpha \in \CM_{R,\mS}(f,f')$. It follows from the equalities $f'=\alpha f$, $A(\gp) = \rann_A(\gp)$ and $A'(\gp) = \rann_{A'}(\gp)$ for all $\gp \in \min (R)$ that  $\alpha =(\alpha_\gp)_{\gp \in \min (R)}$ where   the maps
$$\alpha_\gp : A(\gp)\ra A'(\gp)$$
are monomorphisms  for all $\gp \in \min (R)$ (since the rings $A(\gp)$ are simple). Hence, so is the map $\alpha: A\ra A'$.
\end{proof}

\begin{definition}
Suppose that $R$ is a semiprime ring with  $|\min (R)|<\infty$.  We say that natural $R$-monomorphisms $f\ra A=\prod_{\gp \in \min (R)}A(\gp)$  and $f'\ra A'=\prod_{\gp \in \min (R)}A'(\gp)$ are  {\bf $M$-equivalent} and write $f\stackrel{M}{\sim}f'$   if the rings  $A(\gp)$ and $A'(\gp)$ are Morita equivalent  for all $\gp \in \min (R)$. Let $[f]_M^\sim$ be the $M$-equivalence class of the element $f$ in $\CN_{R, \mS}$.
\end{definition}

The relation $f\stackrel{M}{\sim}f'$ is an equivalence relation which is called the {\bf $M$-equivalence}.  If $f\stackrel{M}{\sim}f'$   then the rings $A$ and $A'$ are Morita equivalent but not vice versa, in general. So, the $M$-equivalence is a finer equivalence relation than the Morita equivalence. The reason for that is that the $M$-equivalence  is the Morita equivalence that preserves/respects an additional structure, namely the product $\prod_{\gp \in \min (R)}A(\gp)$. \\

{\bf Structure of embeddings  $\CM_{R, \mS}$.} 
 For a simple Artinian ring $\L\simeq M_n(D)$, where $D$ is a division ring, the  natural number $s(\L):=n$ is called the {\bf matrix size} of $\L$. The matrix size of $\L$ is the length of a unique (up to isomorphism) simple  left/right $\L$-module. Let  $A=\prod_{i=1}^mA_i$ be  a semisimple Artinian rings where $A_i$ \ae simple Artinian ring. The sum $s(A)=\sum_{i=1}^m s(A_i)$ is called the  {\bf matrix size} of $A$.

Theorem \ref{B9Sep23} is one of the key results of the paper. It captures the essence of the structure of embeddings $\CM_{R, \mS}$. It reveals that left $R$-collections play  the central role in understanding of the structure of embeddings $\CM_{R, \mS}$. It shows that for each  embedding in $f\in \CM_{R, \mS}$ there is a smaller embedding which is  a collection embedding $\phi_V$, and it produces {\em explicit} morphisms in $\Mor (\CM_{R, \mS})$ that proves the smallness.  So, the proof is `constructive'. 

\begin{theorem}\label{B9Sep23}
Suppose that $R$ is a semiprime ring and  $\CM_{R,\mS}\neq \emptyset$. Then 
\begin{enumerate}
\item $|\min (R)|<\infty$.

\item Suppose that  $R\stackrel{f}{\ra}\prod_{i=1}^sA_i\in \Irr\,\CM_{R, \mS}$ is an  irredundant embedding where $A_i\simeq M_{n_i}(D_i)$ are simple Artinian rings and $D_i$ are division rings. Then 
\begin{enumerate}
\item  $s\leq |\min (R)|$,  
\item there exist a  left collection $V\in \CV_l(R)$, a monomorphism  $R\stackrel{g}{\ra}B\times B'\in \CM_{R, \mS}$ where  $R\stackrel{\phi_V}{\ra} B\in  \CM_{R, \mS}$ is the embedding of the collection $V$ and $B'$ is a semisimple Artinian ring such that $$\pr_B\in \CM_{R, \mS}(g,\phi_V),$$  where   $\pr_B:B\times B'\ra B$ is the projection onto $B$, a monomorphism 
 $\alpha  \in \CM_{R, \mS}(g, f)$,   i.e. there is a commutative diagram
 $$\begin{array}[c]{ccc}
 &  & A\;\;\;\\
& {\nearrow}^{f} &{\uparrow}^{\alpha}\\
R&\stackrel{g}{\ra} &B\times B'\\
&{\searrow}^{\phi_V}&{\;\;\;\downarrow}^{\pr_B}\\
&  &B\;\;
\end{array}$$
and

\item  $V=\{\bM_{ij}\, | \, i=1, \ldots , s; j\in J_i\}$ for some nonempty subsets $J_i=\{ 1, \ldots , m_i\}$ of natural numbers where $m_i\leq n_i$, each ${}_R{\bM_{ij}}_{A_i}$ is a simple subfactor of ${}_R{A_i}_{A_i}$ such that 
the map 
$$\coprod_{i=1}^sJ_i\ra  \min (R), \;\; (i,j)\mapsto \ann_R({}_R\bM_{ij})$$
is a bijection. In particular, $|\min (R)|=\sum_{i=1}^s|J_i|$, 
 $$R\, \stackrel{\phi_V}{\ra} \, B=\prod_{i=1}^s\prod_{j\in J_i}\End ({\bM_{ij}}_{A_i})\simeq  \prod_{i=1}^s\prod_{j\in J_i} M_{l_{ij}}(D_i), \;\; l_{ij}:=l({\bM_{ij}}_{A_i}), $$
 $\sum_{j\in J_i}l_{ij}\leq n_i$ for all $i=1, \ldots , s$, and $s(B)=\sum_{i=1}^s\sum_{j\in J_i}l_{ij}\leq s(A)=\sum_{i=1}^sn_i$.

\item For each $i=1, \ldots, s$, the rings $\End ({\bM_{ij}}_{A_i})\simeq M_{l_{ij}}(D_i) $ and $A_i=M_{n_i}(D_i)$ are Morita  equivalent and  $s(\End ({\bM_{ij}}_{A_i}))=l_{ij}\leq n_i = s(A_i)$ for all $j\in J_i$.
\end{enumerate}
\end{enumerate}
\end{theorem}

\begin{proof} The semisimple Artinian algebra $A=\prod_{i=1}^sA_i$ is an  $(R,A)$-bimodule  of finite length since 
$$l({}_RA_A)\leq l(A_A)<\infty.$$ 
The ring   $A=\prod_{i=1}^sA_i$ is a direct product/sum of  $(R,A)$-bimodules $A_i$.
 For each $i=1,\ldots , s$,  there is a finite  descending chain of $(R,A)$-subbimodules of $A_i$,
$${}_R{A_i}_A={}_R{A_i}_{A_i}=M_{i1}\supset\cdots \supset M_{ij}\supset \cdots \supset M_{im_i}\supset M_{i,m_i+1}=\{0\}, $$
with simple factors $\bM_{ij}:=M_{ij}/M_{i,j+1}$, i.e. 
${}_R{\bM_{ij}}_{A_i}$ are {\em simple}.  Hence, 
$$\ga_{ij}:= \ann_R(\bM_{ij})\in \Spec (R)$$ 
where $\ann_R(\bM_{ij})$ is the annihilator of the left $R$-module $\bM_{ij}$. Let   
 $\ga :=\bigcap_{i=1}^s\bigcap_{j=1}^{m_i}\ga_{ij}$ and  $m=\max \{m_i\, | \, i=1, \ldots , s\}$. Then $\ga^m=\ga^m A=\prod_{i=1}^s\ga^mA_i=\prod_{i=1}^s0=0$, and so 
$$\ga =0$$ since the ring $R$ is a semiprime ring.
Hence, there are subsets $J_i\subseteq \{ 1, \ldots , m_i\}$, where $i=1, \ldots , s$, such that the intersection $$\bigcap_{i=1}^s\bigcap_{j\in J_i}\ga_{ij}=0$$
is irredundant.  For each $\gp \in \min (R)$, $$\bigcap_{i=1}^s\bigcap_{j\in J_i}\ga_{ij}=\{0\}\subseteq \gp,$$ and so $\ga_{ij}\subseteq \gp$ for some $(i,j)$,  i.e. $\ga_{ij}= \gp$, by the minimality of $\gp$. Therefore, $|\min (R)|<\infty$.  Then, by Lemma \ref{b10Sep23}, 
$$\min (R)=\{ \ga_{ij}\, | \, i=1,\ldots , s; j\in J_i\}.$$
Notice that ${A_i}_{A_i}\simeq \bigoplus_{j=1}^{m_i}{\bM_{ij}}_{A_i}=\bM_{i,J_i}\oplus \bM_{i, CJ_i}$ where 
$\bM_{i,J_i}:=\bigoplus_{j\in J_i}{\bM_{ij}}_{A_i}$ and $\bM_{i,CJ_i}:=\bigoplus_{j\in CJ_i}{\bM_{ij}}_{A_i}$ where $CJ_i:=\{1, \ldots , m_i\}\backslash J_i$. Now, $\sum_{j\in J_i}l_{ij}\leq n_i$, where $l_{ij}:=l({\bM_{ij}}_{A_i})$,  and 
$$A_i\simeq \End ({A_i}_{A_i})= \End ((\bM_{i,J_i}\oplus \bM_{i, CJ_i})_{A_i})\supseteq  \prod_{j\in J_i}\End ({\bM_{ij}}_{A_i})\times\End({\bM_{i, CJ_i}}_{A_i})$$
is an embedding of the direct product of rings into the ring $A_i$. 
Consider  the semisimple Artinian rings $$B:= \prod_{i=1}^s\prod_{j\in J_i}\End ({\bM_{ij}}_{A_i})\;\; {\rm and}\;\; B':= \prod_{i=1}^s\prod_{j\in CJ_i}\End ({\bM_{ij}}_{A_i})
$$ 
Hence,  
$$ \alpha : B\times B'\ra \prod_{i=1}^sA_i=A$$
is an embedding of the semisimple Artinian ring $B\times B'$ into $A$. Since $\bigcap_{i=1}^s\bigcap_{j\in J_i}\ga_{ij}=0$, there are monomorphisms (diagonal embeddings) 
$$ R\ra\prod_{i=1}^s\prod_{j\in J_i} R/\ga_{ij}\ra B=\prod_{i=1}^s\prod_{j\in J_i}\End ({\bM_{ij}}_{A_i})\simeq   \prod_{i=1}^s\prod_{j\in J_i} M_{l_{ij}}(D_i).
$$
Hence, $V=\{ {}_R{\bM_{ij}}_{A_i}\, | \, i=1,\ldots , s; j\in J_i\}\in \CV_l(R)$ is a left  $R$-collection for the semisimple Artinian ring $B$  and the monomorphism above
 $$R\, \stackrel{\phi_V}{\ra} \, B=\prod_{i=1}^s\prod_{j\in J_i}\End ({\bM_{ij}}_{A_i})$$
is the embedding of the collection $V$.
 Evidently, $R\stackrel{\phi_V}{\ra} B\in \CM_{R, \mS}$ and $R\stackrel{g}{\ra} B\times B'\in \CM_{R, \mS}$ (via diagonal embeddings), $\alpha  \in \CM_{R, \mS}(g, f)$ is a monomorphism and the projection $\pr_B: B\times B'\ra B$ is an epimorphism such that $\pr_B\in \CM_{R, \mS}(g,\phi_V)$.
 
  By the assumption, the embedding $R\stackrel{f}{\ra}A =\prod_{i=1}^sA_i$ is an irredundant embedding. This implies that  $J_i\neq \emptyset$ for all $i=1, \ldots , s$ (otherwise $J_i=\emptyset$ for some $i$ and $R\stackrel{f}{\ra}A =\prod_{j\neq i}^sA_j$ is an embedding, a contradiction). Therefore, $s=|\min (R)|$.
  Now, the statement (d) is obvious. 
\end{proof}

\begin{corollary}\label{a13Sep23}
A semiprime ring with infinitely many minimal primes cannot be embedded into a semisimple Artinian ring. 
\end{corollary}

\begin{proof} The corollary follows from Theorem \ref{B9Sep23}.(1).
\end{proof}

\begin{corollary}\label{a14Sep23}
Suppose that $R$ is a semiprime ring with $|\min (R)|<\infty$ and $R\stackrel{f}{\ra}\prod_{i=1}^sA_i\in \CM_{R, \mS}$ is an   embedding where $A_i$ are simple Artinian rings.  If $s>|\min (R)|$ then the   embedding $f$ is redundant.
\end{corollary}

\begin{proof} By Theorem \ref{B9Sep23}.(1), $|\min (R)|<\infty$. Suppose that the embedding $f$ is irredundant. Then  $s\leq |\min (R)|$ (Theorem \ref{B9Sep23}.(2a)), a contradiction. Therefore, the   embedding $f$ is redundant.
\end{proof}

For a not necessarily irredundant embedding in $\CM_{R, \mR}$, Corollary \ref{a9Sep23} displays a diagram  of how to produce a smaller collection embedding. The result is a direct corollary of  Theorem \ref{B9Sep23}.(2a). 

\begin{corollary}\label{a9Sep23}
We keep the notation of Theorem \ref{B9Sep23}.(2). Suppose that $R$ is a semiprime ring with $|\min (R)|<\infty$ and  $R\stackrel{f'}{\ra}A'\in \CM_{R, \mS}$. Then $A'=A\times A''$ for some semisimple Artinian rings $A$ and $A''$ such that 
  $$R\stackrel{f}{\ra}A\in \Irr\, \CM_{R, \mS}$$ is  an irredundant embedding, $f=\pr_A f'$, 
 $\pr_A\in \CM_{R, \mS}(f',f) $ is the projection of $A'$ onto $A$.  There exist  a  left collection $V\in \CV_l(R)$, a monomorphism  $R\stackrel{g}{\ra}B\times B'\in \CM_{R, \mS}$ where  $R\stackrel{\phi_V}{\ra} B\in  \CM_{R, \mS}$ is the embedding of the collection $V$ and $B'$ is a semisimple Artinian ring such that 
  $$\pr_B\in \CM_{R, \mS}(g,\phi_V),$$ where $\pr_B:B\times B'\ra B$ is the projection onto $B$,  $\alpha  \in \CM_{R, \mS}(g, f)$ is a monomorphism, i.e. there is a commutative diagram
 $$\begin{array}[c]{ccc}
  &  & A'\;\;\;\\
& {\nearrow}^{f'} &{\downarrow}^{\pr_A}\\
 &  & A\;\;\;\\
& {\nearrow}^{f} &{\uparrow}^{\alpha}\\
R&\stackrel{g}{\ra} &B\times B'\\
&{\searrow}^{\phi_V}&{\;\;\;\downarrow}^{\pr_B}\\
&  &B\;\;
\end{array}$$

\end{corollary}

\begin{proof}  Suppose that  $R\stackrel{f'}{\ra}A'\in \CM_{R, \mS}$. Then $A'=A\times A''$ for some semisimple Artinian rings $A$ and $A''$ such that 
  $R\stackrel{f}{\ra}A\in \CM_{R, \mS}$ is  an irredundant embedding, $f=\pr_A f'$ and  
 $\pr_A\in \CM_{R, \mS}(f',f) $ is the projection of $A'$ onto $A$. Now, the result  follows from Theorem \ref{B9Sep23}.(2b).
 \end{proof}
 
\begin{definition}
Suppose that $R$ is a semiprime ring with  $|\min (R)|<\infty$.  We say that an irredundant  monomorphism $f: R\ra A=\prod_{i=1}^sA_i\in \Irr\, \CM_{R, \mS}$ is  {\bf of maximal size} or {\bf has the maximal size} if  $s=|\min (R)|$. 
\end{definition}
Every  natural $R$-monomorphism  is an irreducible  $R$-monomorphism of maximal size (Proposition \ref{a16Sep23}.(1)).

\begin{corollary}\label{b16Sep23}
 We keep the notation of Theorem \ref{B9Sep23}. Suppose that $R$ is a semiprime ring with $s=|\min (R)|<\infty$.
 Suppose that the embedding  $f: R\ra A=\prod_{i=1}^sA_i\in \Irr\, \CM_{R,\mS}$ is of maximal size. Then  $J_i=\{(i,1)\}$ for all $i=1, \ldots , s$, $\min (R)=\{ \ann_R(\bM_{i1})\, | \, i=1, \ldots , s\}$, 
$$R\stackrel{\phi_V}{\ra} B=\prod_{i=1}^s\End ({\bM_{i1}}_{A_i})\simeq \prod_{i=1}^sM_{l_{i1}}(D_i)\;\; {\rm  and}\;\;l_{i1}\leq n_i\;\; {\rm  for\; all}\;\; i=1, \ldots , s,$$ $\phi_V\leq f$, and $[f]_M^\sim =[\phi_V]_M^\sim $. So, for every irredundant $R$-monomorphism of maximals size $f$, there is a smaller, natural, $M$-equivalent  collection  $R$-monomorphism 
$\phi_V$ such that the matrix size of each simple artinian ring of $\phi_V$ does not exceed the matrix size of the corresponding simple Artinian ring of $f$.

\end{corollary} 

\begin{proof}  The corollary  follows from the inclusion $\CM_{R, \mS}^c\subseteq \CN_{R, \mS}$ (Proposition  \ref{a16Sep23}.(1),  Theorem \ref{B9Sep23}.(2) and the equality $s=|\min (R)|$.
\end{proof}

 {\bf Description of minimal embeddings in  an $M$-equivalence class of embeddings.}
\begin{definition}
Suppose that $R$ is a semiprime ring with $|\min (R)|<\infty$. A natural  embedding $R\stackrel{f}{\ra}A=\prod_{\gp \in \min (R)}A(\gp)\in \CN_{R,\mS}$ is called an {\bf elementary embedding} if
${}_R{A(\gp)}_{A(\gp)}$ is a simple bimodule for all $\gp \in \min (R)$. The class of  elementary embedding is denoted by $\CN_{R, \mS}^e$.
\end{definition}

$\bullet$ {\em Every elementary embedding $R\stackrel{f}{\ra}A=\prod_{\gp \in \min (R)}A(\gp)\in \CN_{R,\mS}^e$ is a collection embedding}, 
\begin{equation}\label{f=pVf}
f=\phi_{V_f},
\end{equation}
{\em for the left $R$-collection $V_f:=\{ V_f(\gp ):= A(\gp)\, | \, \gp \in  \min (R)\}\in \CV_l(R)$ since $\End (A(\gp)_{A(\gp)})=A(\gp)$ for all} $\gp \in\min (R)$.

Notice that 
\begin{equation}\label{f=pVf-1}
\CN_{R, \mS}^e\subseteq \CM_{R, \mS}^c\subseteq \CN_{R, \mS}\subseteq \Irr \CM_{R, \mS}.
\end{equation}

\begin{definition}
Let $\CC$ be a class of  morphisms in the category $\CM_{R,\mS}$. We say that the class $\CC$ satisfies the {\bf ascending chain condition} (a.c.c.) if any chain of morphisms in $\CC$, 
$$\cdots \ra A_2\stackrel{f_2}{\ra}A_1\stackrel{f_1}{\ra}A_0,$$
stabilizes, meaning all but finitely many homomorphisms $f_i$ are isomorphisms.
\end{definition}

\begin{theorem}\label{15Sep23}
Suppose that $R$ is a semiprime ring with $|\min (R)|<\infty$ and $C$ be an $M$-equivalence class in $\CN_{R, \mS}$. Then 
\begin{enumerate}

\item For every $f\in C$, there is an element $f'\in C\cap \CN_{R,\mS}^e$ such that $f'\leq f$.

\item $\min (C)\subseteq C\cap \CN_{R,\mS}^e\neq \emptyset$, i.e. the minimal embeddings  in the class $C$ (if exist) are  elementary embeddings that belong to the class $C$.

\item If  $C\cap \CN_{R,\mS}^e$ satisfies the a.c.c. then $\min (C)\neq \emptyset$.

\end{enumerate}
\end{theorem}

\begin{proof} 1.  Suppose that $f\in C$. By Corollary \ref{b16Sep23}, there is a smaller, natural, $M$-equivalent  collection  $R$-monomorphism 
$\phi_{V_1}\in C$ such that the size of each simple artinian ring of $\phi_V$ does not exceed the size of the corresponding simple Artinian ring of $f$. Repeating the same  step several more times we obtain a chain strictly decreasing embeddings in the class $C$, say
$$f\geq \phi_{V_1}> \phi_{V_2}> \cdots ,$$
 that clearly must terminate (since a finite number of sizes cannot decrease indefinitely) at an elementary embedding (otherwise we could have found a strictly smaller embedding, see the proof of Theorem \ref{B9Sep23}.(2)). 

2.  Statement 2 follows from statement 1.

3. Straightforward.
\end{proof}

\begin{corollary}\label{e15Sep23}
Suppose that $R$ is a semiprime ring with $|\min (R)|<\infty$. Then $$\min \CM_{R,\mS}=\min \CN_{R,\mS}=\min \CN_{R,\mS}^e.$$
\end{corollary}

\begin{proof} The equalities  $\min \CM_{R,\mS}=\min \CN_{R,\mS}$ and  $\min \CM_{R,\mS}=\min \CN_{R,\mS}^e$ follow from Corollary \ref{b16Sep23}   and  Theorem \ref{15Sep23}.(1), respectively. 
\end{proof}

Suppose that $R$ is a prime ring. Then automatically $|\min (R)|=1<\infty$. Clearly, 

\begin{equation}\label{RprCN}
\CN_{R, \mS}=\{ R\ra A\, |\, A\in \mathfrak{S}, {}_RA_A\;\; {\rm is\; simple}\}.
\end{equation}
Recall that $ \mathfrak{S}$ is the class of simple Artinian rings. Since  $|\min (R)|=1<\infty$, {\em the $M$-equivalence is the same as the Morita equivalence}.

\begin{corollary}\label{a15Sep23}
Suppose that $R$ is a prime ring  and $C$ be the  Morita equivalence class in $\CN_{R, \mS}$. Then 
\begin{enumerate}

\item For every $f\in C$, there is an element $f'\in C\cap \CN_{R,\mS}^e$ such that $f'\leq f$.

\item $\min (C)\subseteq C\cap \CN_{R,\mS}^e\neq \emptyset$, i.e. the minimal embeddings  in the class $C$ (if exist) are  elementary embeddings that belong to the class $C$.

\item If  $C\cap \CN_{R,\mS}^e$ satisfies the a.c.c. then $\min (C)\neq \emptyset$.

\end{enumerate}
\end{corollary}

\begin{proof} The corollary is a particular case of Theorem \ref{15Sep23}.
\end{proof}

\begin{definition}
Let $\mS'$ be a subcategory of $\mS$. We say that  $\mS'$ is {\bf closed under the Morita equivalence} if $A\in \mS'$, $A'\in  \mS$ and $A'\stackrel{M}{\sim}A$ then $A'\in  \mS'$. 
\end{definition}

 The category $\mS_f$  of finite dimensional semisimple algebras (over the fixed ground field) is  closed under the Morita equivalence. 
 For category  $\mS_f$  all the previous results hold. In particular, Corollary 
\ref{b15Sep23} holds.

\begin{corollary}\label{b15Sep23}
Suppose that $R$ is a semiprime ring with $|\min (R)|<\infty$, $\mS_f$ be the category  of finite dimensional semisimple algebras (over the fixed ground field) and $C$ be an $M$-equivalence class in $\CN_{R, \mS_f}$. Then 
\begin{enumerate}

\item For every $f\in C$, there is an element $f'\in C\cap \CN_{R,\mS_f}^e$ such that $f'\leq f$.

\item $\emptyset\neq \min (C)\subseteq C\cap \CN_{R,\mS_f}^e$, i.e. the minimal embeddings  in the class $C$ are  elementary embeddings that belong to the class $C$.


\item $\min (\CM_{R,\mS_f})\neq \emptyset$ and $\min (\CM_{R,\mS_f})\subseteq \CN_{R,\mS_f}^e$.

\end{enumerate}
\end{corollary}

\begin{proof} The category $\mS_f$ satisfies the  a.c.c., and so the corollary follows from Theorem 
\ref{15Sep23}. 
\end{proof}

\begin{corollary}\label{c15Sep23}
Suppose that $R$ is a prime ring  and $C$ be the  Morita equivalence class in $\CN_{R, \mS_f}$. Then 
\begin{enumerate}

\item For every $f\in C$, there is an element $f'\in C\cap \CN_{R,\mS_f}^e$ such that $f'\leq f$.

\item $\emptyset\neq \min (C)\subseteq C\cap \CN_{R,\mS_f}^e$, i.e. the minimal embeddings  in the class $C$ are  elementary embeddings that belong to the class $C$.


\item $\min (\CM_{R,\mS_f})\neq \emptyset$ and $\min (\CM_{R,\mS_f})\subseteq \CN_{R,\mS_f}^e$.

\end{enumerate}
\end{corollary}

\begin{proof}  The corollary is a particular case of Corollary \ref{b15Sep23}.
\end{proof}

{\bf Elementary left $(R,A)$-collections.} Suppose that $R$ is a semiprime ring with $|\min (R)|<\infty$. Let $V$ be a left $(R, A)$-collection where
 $$A=\prod_{\gp \in \min (R)}A(\gp)\;\; {\rm and  }\;\; \phi_V:R\ra \prod_{\gp \in \min (R)}A_V(\gp).$$
  Recall that the semisimple Artinian rings $A(\gp)$ and $A_V(\gp)$ are Morita equivalent (Lemma \ref{b14Sep23}.(2)). Isomorphic  rings are Morita equivalent but not vice versa, in general. 

\begin{definition}
Suppose that $R$ is a semiprime ring with $|\min (R)|<\infty$. The left $(R,A)$-collection $V$ is called  {\bf elementary} if  the embedding $\phi_V\in \CM_{R, \mS}$   is 
 elementary, i.e. for each $\gp \in \min (R)$, 
the $(R/\gp, A_V(\gp))$-bimodule $A_V(\gp)$ is simple.
Let $ \CV_l^e(R)$ be   the class of elementary left $R$-collections. 
\end{definition}


\begin{lemma}\label{d15Sep23}
Suppose that $R$ is a semiprime ring with $|\min (R)|<\infty$. Then the left $(R,A)$-collection $V$  is an  elementary  left $(R,A)$-collection iff  $A_V^c:=\{ A_V(\gp)\}$ is a left $(R,A_V)$-collection; and in this case $\phi_{A_V^c}\simeq \phi_V$.
\end{lemma}

\begin{proof} Straightforward, see (\ref{f=pVf}).
\end{proof}

{\bf Embeddings of semiprime  rings  into    semiprimary rings.} Recall that  $\mathcal{S}$ is the class of semiprimary rings and $\mS\subseteq \mathcal{S}$. 
 A finite direct product of semiprime  rings is a semiprime  rings. Every semiprimary ring is a unique  finite direct product of semiprimary rings none of which is a direct product of two rings. 

\begin{lemma}\label{a19Sep23}
Let $R\stackrel{f}{\ra} A\in \CM_{R,\mathcal{S}}$. Then $R\cap \rad (A)=0$, $R\stackrel{\of}{\ra} A/\rad (A)\in \CM_{R,\mS}$ and $\of\leq f$.
\end{lemma}

\begin{proof} The ring $R$ is a semiprime ring which can be  identified  with its image in the ring $A$  via the monomorphism $f$. The radical $\rad (A)$ is a nilpotent ideal. Hence the intersection $R\cap \rad (A)$ is a nilpotent ideal of $R$. Then  $$R\cap \rad (A)=0$$ since the ring $R$ is a semiprime ring, and so 
 $R\stackrel{\of}{\ra} A/\rad (A)\in \CM_{R,\mS}$
  and the epimorphism $\pi :A\ra A/\rad (A)$ belongs to  $\mathcal{S}(f,\of )$. Hence, $\of\leq f$.
\end{proof}

\begin{corollary}\label{b13Sep23}
A semiprime ring with infinitely many minimal primes cannot be embedded into a semiprimary ring. 
\end{corollary}
\begin{proof} Suppose that $R\stackrel{f}{\ra} A\in \CM_{R,\mathcal{S}}$. By Lemma \ref{a19Sep23},  $R\stackrel{\of}{\ra} A/\rad (A)\in \CM_{R,\mS}$. This contradicts to Corollary \ref{a13Sep23}. 
\end{proof}


\section{The largest left quotient rings and their  embeddings  into  semisimple Artinian or  semiprimary rings.} \label{LLQRINGS} 

At the beginning of the section, we collect results on localizations and the largest left quotient ring that are used in proofs of the next two sections. 

For a semiprime ring $R$ with finitely many minimal primes, Theorem \ref{A21Sep23} (resp.,  Corollary \ref{aA21Sep23}) shows that the ring $R$ and all its left localizations at regular left  (resp., right ) denominator sets  have the same elementary embeddings. So, in order  to describe all elementary embeddings  for these rings it suffices to do it for any of them. As localizations make the structure of rings `simpler' it suffices to describe the elementary embeddings for the largest left (resp., right) quotient ring of the ring. 

Theorem \ref{B21Sep23} describes the minimal embeddings for all semiprime left Goldie rings and their localizations.  In this section, a proof of Theorem \ref{21Sep23}, which  is a criterion for a semiprime ring $R$ to be embeddable into  semisimple Artinian    rings, is given.\\
 
 {\bf The largest regular left Ore set and the largest left quotient ring of a ring}. Let $R$ be a ring. A {\bf multiplicatively closed subset} $S$ of $R$ or a {\bf
 multiplicative subset} of $R$ (i.e., a multiplicative sub-semigroup of $(R,
\cdot )$ such that $1\in S$ and $0\not\in S$) is said to be a {\bf
left Ore set} if it satisfies the {\bf left Ore condition}: for
each $r\in R$ and $s\in S$, 
  $$Sr\bigcap Rs\neq \emptyset.$$
Let $\Ore_l(R)$ be the set of all left Ore sets of $R$.
  For  $S\in \Ore_l(R)$, $\ass_l (S) :=\{ r\in
R\, | \, sr=0 \;\; {\rm for\;  some}\;\; s\in S\}$  is an ideal of
the ring $R$.

A left Ore set $S$ is called a {\bf left denominator set} of the
ring $R$ if $rs=0$ for some elements $ r\in R$ and $s\in S$ implies
$tr=0$ for some element $t\in S$, i.e., $r\in \ass (S)$. Let
$\Den_l(R)$ be the set of all left denominator sets of $R$. For
$S\in \Den_l(R)$, the ring $$S^{-1}R=\{ s^{-1}r\, | \, s\in S, r\in R\}$$
 is called  the {\bf left localization} of the ring $R$ at $S$ (the {\bf
left quotient ring} of $R$ at $S$). In Ore's method of localization one can localize  precisely at left denominator sets.
 In an obvious way, the {\bf right Ore condition}, {\bf  right Ore} and {\bf right  denominator sets} of $R$ are defined and they are denoted by $\Ore_r(R)$ and $\Den_r(R)$, respectively. For
$S\in \Den_r(R)$, the ring $$RS^{-1}=\{ rs^{-1}\, | \, s\in S, r\in R\}$$
 is called  the {\bf right localization} of the ring $R$ at $S$ (the {\bf
right quotient ring} of $R$ at $S$). For  $S\in \Ore_r(R)$, $\ass_r (S) :=\{ r\in
R\, | \, rs=0 \;\; {\rm for\;  some}\;\; s\in S\}$  is an ideal of
the ring $R$.
 
 A left and right Ore or denominator set is  called  an {\bf Ore set}  or a {\bf denominator set} and these sets are denoted by $\Ore (R)$ and $\Den (R)$, respectively. If $S\in \Den (R)$ then $$S^{-1}R\simeq RS^{-1}.$$ 
For an ideal $\ga$ of $R$, let $\Den_*(R,\ga):=\{ S\in \Den_*(R)\, | \, \ass_* (S)=\ga\}$ where $*\in \{ l,r,\emptyset\}$. 

In general, the set $\CC_R$ of regular elements of a ring $R$ is
neither left nor right Ore set of the ring $R$ and as a
 result neither left nor right classical  quotient ring ($Q_{l,cl}(R):=\CC_R^{-1}R$ and
 $Q_{r,cl}(R):=R\CC_R^{-1}$) exists.
 There  exists the largest (w.r.t. $\subseteq$)
 regular left Ore set $S_l(R)$, \cite{larglquot}. This means that the set $S_l(R)$ is an Ore set of
 the ring $R$ that consists
 of regular elements (i.e., $S_l(R)\subseteq \CC_R$) and contains all the left Ore sets in $R$ that consist of
 regular elements. Also, there exists the largest regular  right (respectively, left and right) Ore set  $S_r(R)$ (respectively, $S_r(R)$, $S_{l,r}(R)$) of the ring $R$.
 In general,  the sets $\CC_R$, $S_l(R)$, $S_r(R)$ and $S_{l,r}(R)$ are distinct, for example,
 when $R= \mI_1= K\langle x, \der , \int\rangle$  is the ring of polynomial integro-differential operators  over a field $K$ of characteristic zero,  \cite{Bav-intdifline}. In  \cite{Bav-intdifline},  these four sets are explicitly described  for $R=\mI_1$.

\begin{definition} (\cite{Bav-intdifline, larglquot}.)    The ring
$$Q_l(R):= S_l(R)^{-1}R$$ (respectively, $Q_r(R):=RS_r(R)^{-1}$ and
$Q(R):= S_{l,r}(R)^{-1}R\simeq RS_{l,r}(R)^{-1}$) is  called
the {\bf largest left} (respectively, {\bf right and two-sided})
{\bf quotient ring} of the ring $R$.
\end{definition}
 In general, the rings $Q_l(R)$, $Q_r(R)$ and $Q(R)$
are not isomorphic, for example, when $R= \mI_1$, \cite{Bav-intdifline}.

Let $R$ be a ring. We say that two left localizations of $R$  are {\bf equal} and write $S^{-1}R = S'^{-1}R$ if the map $ S^{-1}R \ra S'^{-1}R$, $s^{-1}r\mapsto s^{-1}r$, is a well-defined isomorphism. This isomorphism is an $R$-isomorphism.   Then the map $S'^{-1}R\ra S^{-1}R$, $s'^{-1}r\mapsto s'^{-1}r$, is also a ring  $R$-isomorphism. So, two localizations are equal iff there is an $R$-{\em isomorphism} between them. So, the relation of `equality' is an equivalence relation on the set of left localizations of the ring $R$. The set of all the equivalence classes is denoted by $\Loc_l(R)$.
Clearly, $S^{-1}R = S'^{-1}R$ iff $\ass_l (S)=\ass_l (T)$, $ \frac{s}{1}\in (S'^{-1}R)^\times$ for all $s\in S$ and $ \frac{s'}{1}\in (S^{-1}R)^\times$ for all $s'\in S'$.

 The next
theorem gives various properties of the ring $Q_l(R)$. In
particular, it describes its group of units.

\begin{theorem}\label{4Jul10}
(\cite[Theorem 2.8]{larglquot}.)
\begin{enumerate}
\item $ S_l (Q_l(R))= Q_l(R)^\times$ {\em and} $S_l(Q_l(R))\cap R=
S_l(R)$.
 \item $Q_l(R)^\times= \langle S_l(R), S_l(R)^{-1}\rangle$, {\em i.e., the
 group of units of the ring $Q_l(R)$ is generated by the sets
 $S_l(R)$ and} $S_l(R)^{-1}:= \{ s^{-1} \, | \, s\in S_l(R)\}$.
 \item $Q_l(R)^\times = \{ s^{-1}t\, | \, s,t\in S_l(R)\}$.
 \item $Q_l(Q_l(R))=Q_l(R)$.
\end{enumerate}
\end{theorem}

 A ring $R$ has {\bf
finite left rank} (i.e. {\bf finite left uniform dimension}) if
there are no infinite direct sums of nonzero left ideals in $R$. 


A ring $R$ is called a {\bf left (right) Goldie ring} if $R$ has finite left (right) uniform dimension and $R$ satisfies the ascending chain condition (the a.c.c.) on left (right) annihilators.

The next theorem is a semisimplicity criterion for the ring
$Q_l(R)$  (statements 2-5 are  Goldie's
Theorem).

\begin{theorem}\label{5Jul10}
(\cite[Theorem 2.9]{larglquot}.) The following properties of a ring $R$ are equivalent:
\begin{enumerate}
\item  $Q_l(R)$ is a semisimple ring. \item $Q_{l, cl}(R)$  exists
and is a semisimple ring. \item $R$ is a left order in a
semisimple ring. \item $R$ has finite left rank, satisfies the
ascending chain condition on left annihilators and is a semi-prime
ring. \item A left ideal of $R$ is essential iff it contains a
regular element.
\end{enumerate}
If one of the equivalent conditions holds then $S_0(R) = \CC_R$ and
$Q_l(R) = Q_{l,cl}(R)$.
\end{theorem}

{\bf Description of the set of minimal primes of a localization of a semiprime ring at a regular denominator set.} 
 If $T$ is an ideal of $R$ then the sets $\lann_R(T)$ and $\rann_R(T)$ are also ideals. Ideals of that kind are called  {\bf annihilator ideals}.  A submodule of a module is called an {\bf essential submodule} if it meets all the nonzero submodules of the module. 
In  a semiprime ring   the left annihilator of an ideal  is equal to  its  right annihilator and vice versa. 

Theorem \ref{A10Sep23}.(2)  provides  an explicit description of  the set  of minimal primes of localizations of  semiprime  rings provided that they  have only finitely many minimal primes. Theorem \ref{A10Sep23}.(2) is one of the  key facts in  understanding  of  the structure of  embeddings of localizations of semiprime  rings  into  semisimple or  semiprimary rings.

\begin{theorem}\label{A10Sep23}
(\cite{MinPrimeLocRings})  Let $R$ be a semiprime ring and $S\in \Den_l(R, 0)$. Then 
\begin{enumerate}

\item The ring $S^{-1}R$ is a semiprime ring.

\item If, in  addition,  $|\min (R)|<\infty$ then $\min (S^{-1}R)=\{ S^{-1}\gp\, | \, \gp \in \min (R)\}$, i.e. the map $\min (R)\ra \min (S^{-1}R)$, $\gp \mapsto S^{-1}\gp$ is a bijection and $|\min (S^{-1}R)|=|\min (R)|$.

\item If, in  addition,  $|\min (R)|<\infty$ then
$\pi_\gp (S)\in \Den_l(R/\gp , 0)$ and  $\pi_\gp (S)^{-1}(R/\gp)\simeq S^{-1}R/S^{-1}\gp \simeq S^{-1}(R/\gp)$ where $\pi_\gp : R\ra R/\gp$, $r\mapsto r+\gp$.
\end{enumerate}
\end{theorem}

For a semiprime ring $R$ with $|\min (R)|<\infty$, Corollary \ref{c10Sep23} describes the set $\min (Q_l(R))$. 

\begin{corollary}\label{c10Sep23}
Let $R$ be a semiprime ring. Then 
\begin{enumerate}

\item The ring $Q_l(R)$ is a semiprime ring.

\item If, in  addition,  $|\min (R)|<\infty$.  Then 
\begin{enumerate}

\item $\min (Q_l(R))=\{ S_l(R)^{-1}\gp\, | \, \gp \in \min (R)\}$.

\item For all $\gp \in \min (R)$, $\pi_\gp (S_l(R))\in \Den_l(R/\gp , 0)$ and  $$\pi_\gp (S_l(R))^{-1}(R/\gp)\simeq S_l(R)^{-1}R/S_l(R)^{-1}\gp\simeq Q_l(R)/S_l(R)^{-1}\gp \simeq S_l(R)^{-1}(R/\gp)$$ where $\pi_\gp : R\ra R/\gp$, $r\mapsto r+\gp$.

\item For all $\gp \in \min (R)$, $\pi_\gp (S_l(R))\subseteq S_l(R/\gp)$ and $Q_l(R)/S_l(R)^{-1}\gp \subseteq Q_l(R/\gp)$.

\end{enumerate}
\end{enumerate}
\end{corollary}

\begin{proof} 1 and 2. By the definition, $Q_l(R)=S_l(R)^{-1}R$ and $S_l(R)\in \Den_l(R, 0)$, and statement 1 and 2 follow from Theorem  \ref{A10Sep23}.(1,2).

3. Statement 3 is a particular case of Theorem \ref{A10Sep23}.(3). 

4. By the definition, $S_l(R/\gp)$ is the largest left denominator set of  the ring $R/\gp$ that consists of regular elements of  $R/\gp$. By statement 3, $$\pi_\gp (S_l(R))\in \Den_l(R/\gp , 0),$$  and so  $\pi_\gp (S_l(R))\subseteq S_l(R/\gp)$ and $Q_l(R)/S_l(R)^{-1}\gp \subseteq Q_l(R/\gp)$.
\end{proof}

{\bf Elementary embeddings of localizations of a semiprime rings.}  Let $R$ be a  ring and  $S\in \Den_l(R)$. For a left  $R$-module $M$, the {\bf set of $S$-torsion elements} of $M$, $\tor_{S}(M):=\{ m\in M\, | \, sm=0\;\; {\rm for \; some}\;\; s\in S\}$, is a submodule of $M$.

 Let $R$ be a semiprime ring with $|\min (R)|<\infty$ and $S\in \Den_l(R, 0)$. 
 By Theorem \ref{A10Sep23}, the ring $S^{-1}R$ is a semiprime ring such that  $\pi_\gp (S)\in \Den_l(R/\gp , 0)$,  $$\min (S^{-1}R)=\{ S^{-1}\gp \, | \, \gp \in \min (R)\}\;\; {\rm and}\;\; S^{-1}R/S^{-1}\gp\simeq \pi_\gp (S)^{-1}(R/\gp). $$
  In particular,  $|\min (S^{-1}R)|=|\min (R)|<\infty$.

Theorem \ref{A21Sep23} (resp.,  Corollary \ref{aA21Sep23}) shows that the ring $R$ and all its left localizations at regular left  (resp., right ) denominator sets  have the same elementary embeddings. So, in order  to describe all elementary embeddings  for these rings it suffices to do it for any of them. As localizations make the structure of rings `simpler' it suffices to describe the elementary embeddings for the largest left (resp., right) quotient ring of the ring.

\begin{theorem}\label{A21Sep23}
Let $R$ be a semiprime ring with $|\min (R)|<\infty$, $R\stackrel{f}{\ra} A= \prod_{\gp \in \min (R)}A(\gp)\in \CN_{R,\mS}^e$ and $S\in \Den_l(R, 0)$. Then 
\begin{enumerate}

\item $f(S)\subseteq A^\times$. 

\item $S^{-1}R\stackrel{S^{-1}f}{\ra} A\in \CN_{S^{-1}R,\mS}^e$ and
$\CN_{S^{-1}R, \mS}^e=\{ S^{-1}R\stackrel{S^{-1}f'}{\ra} A'\, | \, R\stackrel{f'}{\ra} A'\in \CN_{R,\mS}^e\}$.

\item $Q_l(R)\stackrel{S_l(R)^{-1}f}{\ra} A\in \CN_{Q_l(R),\mS}^e$ and
$\CN_{Q_l(R), \mS}^e=\{ Q_l(R)\stackrel{ S_l(R)^{-1}f'}{\ra} A'\, | \, R\stackrel{f'}{\ra} A'\in \CN_{R,\mS}^e\}$.
\end{enumerate}
\end{theorem}

\begin{proof} 1. We identity the ring $R$ with its image in $A$ via the monomorphism $f$.

 (i) $\ga:=\tor_S(A)=0$: By the definition,  $\ga$ is an $(R, A)$-subbimodule of $A$, and 
$$\ga= \prod_{\gp\in \min (R)}\ga(\gp) \;\; {\rm where}\;\; \ga (\gp ):=\ga\cap A(\gp)$$ 
since the ring $A=\prod_{\gp \in \min(R)}A(\gp)$ is a finite direct product of rings $A(\gp)$. Suppose that $\ga\neq 0$. We seek a contradiction. Then $\ga (\gp)\neq 0$ for some $\gp \in \min (R)$, and so 
 $$\ga (\gp)=A(\gp)$$ since $\ga (\gp)$ is a nonzero $(R, A)$-subbimodule of the {\em simple} $(R, A)$-bimodule $A(\gp)$.  Since the right $A(\gp)$-module $A(\gp)$ is finitely generated, $sA(\gp)=0$ for some $s\in S$. 
  Recall that $\gp^c:=\bigcap_{\gq \in \min (R)\backslash \{ \gp\} }\gq\neq 0$ and $\ann_A(\gp^c)=\gp$.   Then $0=0A=\gp\gp^cA$, and so $\gp^cA\subseteq \rann_A (\gp)=A(\gp)$ (since $R\stackrel{f}{\ra}A\in \CN_{R,\mS}^e$). Finally,
 $$0\neq  s\gp^c\subseteq  s\gp^c A\subseteq sA(\gp) =0,$$
a contradiction.   
 
(ii)  $S\subseteq A^\times$: By the statement (i), for each element $s\in S$,  the map $s\cdot : A\ra A$, $a\mapsto sa$ is a right $A$-module monomorphism. Hence, an isomorphism since the length $l(A_A)<\infty$ is  finite. Then $ss'=1$ for some element $s'\in A$, and so  the map $s'\cdot : A\ra A$, $a\mapsto s'a$ is a right $A$-module monomorphism. Hence, an isomorphism, and so $s's''=1$ for    some element $s''\in A$. Then $s'=s^{-1}$, and the statement (ii) follows.  

2. By statement 1, $f(S)\subseteq A^\times$. Then by the universal property of localization, we have the monomorphism  $S^{-1}R\stackrel{S^{-1}f}{\ra} A$. Since $\min (S^{-1}R)=\{ S^{-1}\gp \, | \, \gp \in \min (R)\}$ (Theorem \ref{A10Sep23}.(2)), we have that 
 $$S^{-1}R\stackrel{S^{-1}f}{\ra} A\in \CN_{S^{-1}R,\mS}^e.$$
Therefore, 
$\CN_{S^{-1}R, \mS}^e\supseteq \{ S^{-1}R\stackrel{S^{-1}f'}{\ra} A'\, | \, R\stackrel{f'}{\ra} A'\in \CN_{R,\mS}^e\}$. 

To finish the proof of statement 2,  it remains to show that the reverse inclusion holds.
 Let $S^{-1}R\stackrel{g}{\ra}A\in \CN_{S^{-1}R, \mS}^e$.  Since $l(A_A)<\infty$, for each $\gp \in \min (R)$, the $(R,A)$-bimodule $A(\gp)$ contains  a simple $(R,A)$-subbimodule, say $B(\gp)$. Then $S^{-1}B(\gp)$ is a simple 
$(S^{-1}R,A)$-subbimodule of the simple $(S^{-1}R,A)$-bimodule $S^{-1}A(\gp) =A(\gp)$, and so $S^{-1}B(\gp)=A(\gp)$.
 Since $l(B(\gp)_A)<\infty$, $$B(\gp)=S^{-1}B(\gp)=A(\gp),$$ i.e. $S^{-1}R\stackrel{g}{\ra}A\in \CN_{R, \mS}^e$, as required.


3. Statement 3 is a particular case of statement 2. 
\end{proof}

\begin{corollary}\label{aA21Sep23}
Let $R$ be a semiprime ring with $|\min (R)|<\infty$, $R\stackrel{f}{\ra} A= \prod_{\gp \in \min (R)}A(\gp)\in \CN_{R,\mS}^e$ and $S\in \Den_r(R, 0)$. Then 
\begin{enumerate}

\item $f(S)\subseteq A^\times$. 

\item $RS^{-1}\stackrel{fS^{-1}}{\ra} A\in \CN_{RS^{-1},\mS}^e$ and
$\CN_{RS^{-1}, \mS}^e=\{R S^{-1}\stackrel{f'S^{-1}}{\ra} A'\, | \, R\stackrel{f'}{\ra} A'\in \CN_{R,\mS}^e\}$.

\item $Q_r(R)\stackrel{fS_r(R)^{-1}}{\ra} A\in \CN_{Q_r(R),\mS}^e$ and
$\CN_{Q_r(R), \mS}^e=\{ Q_r(R)\stackrel{ f'S_l(R)^{-1}}{\ra} A'\, | \, R\stackrel{f'}{\ra} A'\in \CN_{R,\mS}^e\}$.
\end{enumerate}
\end{corollary}

\begin{proof} The first condition of Theorem \ref{A21Sep23} is left right symmetric, and the corollary follows (by using opposite rings).
\end{proof}

{\bf Descriptions of the minimal embeddings for the semiprime left or right Goldie rings.} Let us recall the result below that is used in the proof of Theorem \ref{B21Sep23}.

\begin{theorem}\label{VBB-21Sep 23}
(\cite[Theorem 4.1]{NewCrit-S-Simpl-lQuot})
Let $R$ be a ring. The following statements are equivalent:
\begin{enumerate}
\item The ring $Q_l(R)$ is a semisimple Artinian ring. 
\item
\begin{enumerate}
\item The ring $R$ is a semiprime ring.
\item $|\min (R)|<\infty$.
\item For each $\gp \in \min (R)$, the  set set $S_\gp :=\{ c\in R\, | \, c+\gp \in \CC_{R/\gp}\}$ is a left denominator set of $R$ with $\ass_l(S_\gp)=\gp$.
\item For each $\gp \in \min (R)$, the ring $S_\gp^{-1}R$ is a simple Artinian ring.
\end{enumerate}
If the equivalent conditions hold then $\max\Den_l(R)=\{ S_\gp\, | \, \gp \in \min (R)\}$ and 
 $Q_l(R)\simeq\prod_{\gp \in \min (R)}S_\gp^{-1}R$.
\end{enumerate}
\end{theorem}

By Theorem   \ref{5Jul10}, in statement 1 above the ring $Q_l(R)$ is equal to $Q_{l,cl}(R)$. 

Theorem \ref{B21Sep23} describes the minimal embeddings for all semiprime left Goldie rings and their localizations. 

\begin{theorem}\label{B21Sep23}
Let $R$ be a semiprime left Goldie ring. Then 
\begin{enumerate}

\item    $R\stackrel{\s}{\ra} Q_l(R)= Q_{l,cl}(R)= \prod_{\gp \in \min (R)}S_\gp^{-1}R\in \CN_{R,\mS}^e$. 

\item $\min \CM_{R,\mS}=\min \CN_{R,\mS}^e=\{ R\stackrel{\s}{\ra} Q_{l,cl}(R)\}$.   

\item For each $S\in \Den_l(R,0)$, $S^{-1}R\stackrel{\s}{\ra}Q_{l,cl}(R)\simeq Q_{l,cl}(S^{-1}R)\in \CN_{S^{-1}R,\mS}^e$. 

\item For each $S\in \Den_l(R,0)$, $\min \CM_{S^{-1}R,\mS}=\min \CN_{S^{-1}R,\mS}^e=\{ S^{-1}R\stackrel{\s}{\ra} Q_{l,cl}(R)\}$.  

\end{enumerate}
\end{theorem}

\begin{proof} 1. By Theorem  \ref{VBB-21Sep 23},  $R\stackrel{\s}{\ra} Q_{l,cl}(R)\in \CN_{R,\mS}$ and $Q_{l,cl}(R)\stackrel{\tau}{\ra} Q_{l,cl}(R)\in \CN_{R,\mS}^e$. Since $\tau = S_l(R)^{-1}\s$, 
$$R\stackrel{\s}{\ra} Q_{l,cl}(R)\in \CN_{R,\mS}^e,$$ by Theorem \ref{A21Sep23}.(3) (repeat the proof of Theorem \ref{A21Sep23}.(2) where $S=S_l(R)$).

2. By Corollary \ref{e15Sep23}, $\min \CM_{R,\mS}=\min \CN_{R,\mS}^e$.  For each elementary embedding $R\stackrel{f}{\ra}A=\prod_{\gp \in \min (R)}A(\gp)\in\CN_{R,\mS}^e$, there is a commutative diagram

 $$\begin{array}[c]{ccc}
 R&\stackrel{f}{\ra} &A\\
&{\searrow}^{\s}& {\;\;\;\uparrow}^{\CC_R^{-1}f}\\
&  &Q_{l,cl}(R)\;\;
\end{array}$$
Therefore, $\min \CN_{R,\mS}^e=\{ R\stackrel{\s}{\ra} Q_{l,cl}(R)\}$, by statement 1.  

3 and 4. By Theorem \ref{A21Sep23}.(2),   statements 3 and 4 follows from  statements 1 and 2. 
\end{proof}

{\bf Criterion for $\CM_{R, \mS}\neq \emptyset$ where $R$ is a semiprime ring.} Theorem \ref{21Sep23} is a criterion for a semiprime ring $R$ to be embeddable into  semisimple Artinian    rings. The key idea of its proof is Proposition \ref{A22Sep23}.

\begin{proposition}\label{A22Sep23}
Let $R$ be a semiprime ring with $|\min (R)|<\infty$. Then for each elementary embedding $R\stackrel{f}{\ra}A=\prod_{\gp \in \min (R)}A(\gp)\in \CN_{R,\mS}^e$ there is a commutative diagram of ring homomorphisms:

$$\begin{array}[c]{ccccc} 
R & \stackrel{}{\ra} & \prod_{\gp\in \min(R)}R/\gp &  \stackrel{\prod_{\gp \in \min (R)}S_l(R/\gp)^{-1}}{\ra} & \prod_{\gp\in \min(R)}Q_l(R/\gp) \\
&{\searrow}^{f} &{\;\downarrow}  & {\swarrow}_{f'}&\\
&  &A=\prod_{\gp\in \min (R)}A(\gp) & &\;\;
\end{array}$$

\end{proposition}

\begin{proof} The commutativity of the left triangle is obvious. For each $\gp \in \min (R)$, there are monomorphisms $R/\gp\ra A(\gp)$ and $R/\gp\stackrel{S_l(R/\gp)^{-1}}{\ra} Q_l(R/\gp)$. Now, by applying Theorem \ref{A21Sep23}.(1,3) to the elementary embedding $R/\gp\ra A(\gp)$, we obtain the  commutative diagram
$$\begin{array}[c]{ccccc} 
 R/\gp &  \stackrel{S_l(R/\gp)^{-1}}{\ra} & Q_l(R/\gp) \\
{\;\downarrow}  & {\;\;\;\;\;\;  \swarrow}_{f_\gp'}&\\
A(\gp) & &\;\;
\end{array}$$
and so the right triangle of the proposition is also commutative. 
\end{proof}

\begin{proof} {\bf (Proof of Theorem \ref{21Sep23})}    $(1\Leftrightarrow 2)$ The equivalence follows from Theorem \ref{A21Sep23}.(3).

$(1\Rightarrow 3)$  Suppose that $\CM_{R, \mS}\neq \emptyset$. Then, by Theorem \ref{B9Sep23}.(1),  $|\min (R)|<\infty$ and, by Proposition \ref{A22Sep23}, the ring $Q_l(R/\gp)$ is embeddable into a simple Artinian ring for every $\gp \in\min (R)$.

$(3\Rightarrow 1)$ Suppose that  $|\min (R)|<\infty$ and the ring $Q_l(R/\gp)$ is embeddable into a simple Artinian ring for every $\gp \in\min (R)$, say  $Q_l(R/\gp)\ra A(\gp)$. Then 
the embedding $$R\ra \prod_{\gp \in \min (R)}Q_l(R/\gp)\ra  \prod_{\gp \in \min (R)}A(R/\gp)$$ belongs to $\CM_{R,\mS}$.

$(1\Leftrightarrow 4,5)$ The equivalences follows from the equivalences  $(1\Leftrightarrow 2,3)$ and the left-right symmetry of the  first condition.        
 \end{proof}

The commutative diagram in Proposition \ref{A22Sep23} can be refine in order to include  the ring $Q_l(R)$.

\begin{lemma}\label{a22Sep23}
Let $R$ be a semiprime ring with $|\min (R)|<\infty$. Then for each elementary  embedding $R\stackrel{f}{\ra}A=\prod_{\gp \in \min (R)}A(\gp)\in \CN_{R,\mS}^e$, there is a commutative diagram of ring homomorphisms:
$$\begin{array}[c]{ccccccc} 
R & \stackrel{}{\ra} & Q_l(R)&  \stackrel{}{\ra}&\prod_{\gp\in \min(R)}Q_l(R)/S_l(R)^{-1}\gp & \ra & \prod_{\gp\in \min(R)}Q_l(R/\gp) \\
&{\searrow}^{f} &{\;\;\;\;\;\;\;\;\;\downarrow}^{S_l(R)^{-1}f}  & {\;\;\;\;\;\swarrow}& & {\swarrow}_{f'}& \\
&  &A=\prod_{\gp\in \min (R)}A(\gp) & & & & \;\;
\end{array}$$

\end{lemma}

\begin{proof} By Theorem \ref{A21Sep23}.(3),  the left triangle is commutative. For each $\gp \in \min (R)$,  the embedding $R/\gp\ra A(\gp)$  is an elementary embedding (since the embedding $R\stackrel{f}{\ra} A$ is so). By Corollary \ref{c10Sep23}.(2b), 
$$\pi_\gp (S_l(R))\in \Den_l(R/\gp , 0)\;\; {\rm  and}\;\; Q_l(R)/S_l(R)^{-1}\gp\simeq \pi_\gp (S_l(R))^{-1}(R/\gp).$$  
 Now, by applying Theorem \ref{A21Sep23}.(1,3) to the elementary embedding $R/\gp\ra A(\gp)$, we obtain the  commutative diagram
$$\begin{array}[c]{ccccc} 
 R/\gp &  \stackrel{\pi_\gp(S_l(R))^{-1}}{\ra} & \pi_\gp (S_l(R))^{-1}(R/\gp)\simeq Q_l(R)/S_l(R)^{-1}\gp \\
{\;\downarrow}  & {\;\;\;\;\;\;  \swarrow}_{f_\gp'}&\\
A(\gp) & &\;\;
\end{array}$$
and so the middle triangle in the diagram of the proposition is also commutative. 

The inclusion $\pi_\gp (S_l(R))\in \Den_l(R/\gp , 0)$ implies the inclusion  $\pi_\gp (S_l(R))\subseteq S_l(R/\gp)$. Let $U:=\Big(\pi_\gp (S_l(R))^{-1}(R/\gp) \Big)^\times$ and $T_\gp$ be a multiplicative set in the ring $\pi_\gp (S_l(R))^{-1}(R/\gp)$ which is generated by $U$ and $S_l(R/\gp)$. Then 
$$T_\gp\in \Den_l(\pi_\gp (S_l(R))^{-1}(R/\gp) , 0)\;\; {\rm and}\;\; T_\gp^{-1}(R/\gp)\simeq S_l(R/\gp)^{-1}(R/\gp)=Q_l(R/\gp).$$
 Now, by applying Theorem \ref{A21Sep23}.(1,3) to the elementary embedding $\pi_\gp (S_l(R))^{-1}(R/\gp)\ra A(\gp)$, we obtain the  commutative diagram
$$\begin{array}[c]{ccccc} 
 \pi_\gp (S_l(R))^{-1}(R/\gp) &  \stackrel{T_\gp^{-1}}{\ra} & Q_l(R/\gp) \\
{\;\downarrow}  & {\;\;\;\;\;\;  \swarrow}_{f_\gp'}&\\
A(\gp) & &\;\;
\end{array}$$
and so the right triangle of the proposition is also commutative. 
 \end{proof}

{\bf  $\Spec_e(R)$ and the ideal $\kappa_R$.}
\begin{definition}
For a ring $R$,  
$\Spec_e(R):=\{ \gp \in \Spec (R)\, | \, \CM_{R/\gp,\mS}\neq \emptyset\}$ is called the {\bf set  of  embeddable prime ideals} and the set $\CI_e(R):=\{ \ga\, | \, \ga $ is a semiprime ideal of $R$ such that $\CM_{R/\ga, \mS}\neq \emptyset\}$ is called the {\bf set  of  embeddable semiprime ideals} of $R$.
\end{definition}

\begin{lemma}\label{c22Sep23}
Let $R$ be a ring. Then $\CI_e(R)$ consists precisely  of all finite  intersections of embeddable prime ideals of $R$.
\end{lemma}
\begin{proof}
 The lemma follows from Theorem \ref{B9Sep23}.
\end{proof}

\begin{definition}
For a ring $R$, let 
$\kappa_R:=\bigcap_{\gp\in \Spec_e(R)}\gp.$
\end{definition}


\begin{lemma}\label{b22Sep23}
Let $R$ be a ring. Then $\kappa_R$ is a semiprime ideal of the ring $R$ that contains the prime radical $\gn_R$ of the ring $R$.
\end{lemma}
\begin{proof}
 Straightforward.
\end{proof}


\section{New criterion for a ring to have a semisimple Artinian quotient ring}\label{NEWCRIT-THMG}

In this   section,   a proof is given  of Theorem \ref{VVB-23Sep2023}. The  proof is based  on Proposition \ref{A24Sep23} that provides    sufficient conditions for a ring to have a semisimple Artinian quotient ring. \\





{\bf Minimal primes of a semiprime ring and minimal primes of its centre.} Lemma \ref{b23Sep23} shows that, in general, there is 
no relation between minimal primes of a semiprime ring and and its centre. But the situation is somewhat different if, in addition, the ring $R$ is a left/right Goldie ring in the sense that using information about  the centre it is possible to give a new  criterion for the left quotient ring of a ring to be a semisimple Artinian ring (Theorem \ref{VVB-23Sep2023}).

 Let $K(t)$ be be a field of rational functions in the variable $t$ over a field $K$, $A=K(t)[x; \s]$ be a skew polynomial ring where $\s(t)=qt$ where $q\in K^\times$ is a not a root of unity,  $L=K(t)[x^{\pm 1}; \s]$ be a skew polynomial ring, $B=L^n$ be  the direct product of $n\geq 2$ copies of the algebra $L$ and $1=e_1+\cdots +e_n$ be the corresponding sum of central orthogonal idempotents. The algebra $L$ is a Noetherian domain. By Goldie's  Theorem its (left and right) quotient ring $Q(L)$ is a division ring. Hence, $Q(B)\simeq Q(L)^n$. 
Clearly,   $R=K+\sum_{i=1}^n (x^{m_i})e_i$ is a subalgebra of $B$ where $(x^{m_i})$ is an ideal of $L$ generated by the regular normal element $x^{m_i}$ of $A$ where $1\leq m_1\leq\cdots \leq m_n$.  

It is a well known fact that the algebra $L$ is a central simple algebra.

\begin{lemma}\label{a23Sep23}

\begin{enumerate}

\item The algebra $B$ is a semiprime algebra,  $\min (B)=\{ \gq_i:=\sum_{j\neq i}Be_i\, | \, i=1,\ldots , n\}$, and $B/\gq_i\simeq L$ for all $i=1, \ldots , n$.

\item $Z(B)=K^n$.

\item For all $i=1, \ldots , n$, $Z(R)\cap \gp_i=K\sum_{j\neq i}e_i\neq 0$.

\end{enumerate}
\end{lemma}

\begin{proof} 1. The algebra $L$ is a  simple algebra. Hence,  $L$ is a prime algebra, and so $B=L^n$ is a semiprime algebra. 

By the definition of the ideal $\gq_i$, $B/\gq_i\simeq L$ is a domain. Hence, the ideals $\gq_i$ are prime ideals of $B$. The ideals $\gq_i$ are distinct, $\bigcap_{i=1}\gq_i=0$ but $\bigcap_{j\neq i}\gq_j =0$ for all $i=1, \ldots , n$. By Lemma \ref{b10Sep23}, $\min (B)=\{ \gq_i\, | \, i=1,\ldots , n\}$.

2. The algebra $L$ is a central algebra, and so $Z(B)=Z(L^n)=Z(L)^n=K^n$.

3. Statement 3 is obvious. 
\end{proof}

Clearly, $\dim_K(B/R)<\infty$, i.e. the algebra is a large subalgebra of $B$.  Lemma \ref{b23Sep23} shows that the algebra $R$ is a semiprime algebra with finitely many minimal primes but its centre is small and the minimal primes do not meet the centre.

\begin{lemma}\label{b23Sep23}

\begin{enumerate}

\item The algebra $R$ is a semiprime algebra,  $\min (R)=\{ \gp_i:=\sum_{j\neq i}(x^{m_i})e_i\, | \, i=1,\ldots , n\}$, and $R/\gp_i\simeq K+ (x^{m_i})\subseteq L$ for all $i=1, \ldots , n$.

\item For each $l\geq m_n$, the element $\xi=\xi_l:=x^l(e_1+\cdots +e_n)$ is a regular normal element of $R$ such that $R_\xi =L^n$ where $R_\xi$ is the localization of algebra $R$ at the powers of the element $\xi$.

\item $Z(R)=K$.

\item For all $i=1, \ldots , n$, $Z(R)\cap \gp_i=0$.

\item $ Q(R)\simeq Q(L)^n$ where $Q(R)$ is the quotient ring of $R$.
\end{enumerate}
\end{lemma}

\begin{proof} 1. By the definition of the ideal $\gp_i$, $R/\gp_i\simeq K+ (x^{m_i})\subseteq L$ is a domain for all $i=1, \ldots , n$. Hence, the ideals $\gp_i$ are prime ideals of $R$. The ideals $\gp_i$ are distinct, $\bigcap_{i=1}\gp_i=0$ but $\bigcap_{j\neq i}\gp_j =0$ for all $i=1, \ldots , n$. By Lemma \ref{b10Sep23}, $\min (R)=\{ \gp_i\, | \, i=1,\ldots , n\}$.
The zero ideal  $\bigcap_{i=1}\gp_i=0$ is a semiprime ideal, i.e. the algebra $R$ is a semiprime algebra. 

2. By the definition, the element $\xi=\xi_l\in B$ is a unit. Hence, the element $\xi=\xi_l\in R$ 
a regular  element of $R$ which is obviously a normal element. For each $i=1, \ldots , n$, $x^le_i\in R$. Hence, $e_i=x^{-l}\cdot x^le_i\in R_\xi$ for all $i=1, \ldots , n$, and so $R_\xi = L^n=B$.

3. By Lemma \ref{a23Sep23}.(2), $Z(B)=K^n$. Now, 
$Z(R)=R\cap Z(R_\xi)=R\cap Z(B)=R\cap K^n=K$.

4. Since $Z(R)=K$ is a field, statements 4 follows.

5. $Q(R)=Q(R_\xi)=Q(B)=Q(l)^n$. 
\end{proof}

{\bf New criterion for a ring to have a semisimple Artinian quotient ring.} Proposition \ref{A24Sep23} provides sufficient conditions for the left quotient ring $Q_{l,cl}(R)$ of a ring $R$ to  be a semisimple Artinian ring. This result is used in the proof of a new criterion for the left quotient ring $Q_{l,cl}(R)$ to be a semisimple Artinian ring (Theorem \ref{VVB-23Sep2023}).

\begin{proposition}\label{A24Sep23}
Let $R$ be a ring and $S\in \Den_l(R,0)$. 
\begin{enumerate}
\item If the ring $T^{-1}(S^{-1}R)$ is a semisimple Artinian ring for some $T\in \Den_l(S^{-1}R,0)$. Then $Q_{l,cl}(R)\simeq Q_l(R)\simeq Q_l(S^{-1}R)\simeq T^{-1}(S^{-1}R)$ is a semisimple Artinian and the ring $R$ is a semiprime left Goldie ring.
\item If the ring $Q_l(S^{-1}R)$ is a semisimple Artinian ring. Then $Q_{l,cl}(R)\simeq Q_l(R)\simeq Q_l(S^{-1}R)$ is a semisimple Artinian and the ring $R$ is a semiprime left Goldie ring.

\end{enumerate}

\end{proposition} 
\begin{proof} 1. By the assumption $S\in \Den_l(R,0)$ and $T\in \Den_l(S^{-1}R,0)$. So, we have the inclusions of rings $R\subseteq S^{-1}R\subseteq Q_l(S^{-1}R)$.

(i) $T':=R\cap T\in \Den_l(R,0)$ {\rm and} $T'^{-1}R\simeq T^{-1}(S^{-1}R)$: The statement (i) is a particular case of  \cite[Lemma 3.3.(1)]{larglquot}.

(ii) $ Q_l(R)\simeq T'^{-1}R\simeq T^{-1}(S^{-1}R)$ is {\em a semisimple Artinian ring}:    By the statement (i), $T'\in \Den_l(R,0)$, and so  we have the inclusion  $T'\subseteq S_l(R)$ which implies the inclusion  $$T'^{-1}R\subseteq Q_l(R).$$ In fact, the equality holds since the ring $T'^{-1}R\simeq T^{-1}(S^{-1}R)$ is a semisimple Artinian ring, by the statement (i), and therefore $S_l(R)\subseteq \Big(T'^{-1}R\Big)^\times$. This implies the equality.

(iii)  $Q_{l,cl}(R)\simeq Q_l(R)$:   By the statement (ii), the ring $Q_l(R)$ is a semisimple Artinian ring, and the result follows from Theorem \ref{5Jul10}. 

By statements (ii) and (iii), the ring $Q_{l,cl}(R)$ is a semisimple Artinian ring, and so  the ring $R$ is a semiprime left Goldie ring, by Goldie's Theorem.

2. Statement 2 is a particular case of statement 
\end{proof}

\begin{proof} {\bf (Proof of Theorem \ref{VVB-23Sep2023}).}
 $(2\Rightarrow 1)$ Let $Q(Z(R)):=\CC_{Z(R)}^{-1}Z(R)$ be the quotient ring of $Z(R)$.

 (i) {\em The centre $Z(R)$ is a semiprime ring}: Suppose that the  centre $Z(R)$ is not  a semiprime ring. Then it contains a nonzero nilpotent ideal, say $\ga$. Then $R\ga$ is a nonzero nilpotent ideal of the semiprime ring $R$, a contradiction. Therefore, the centre $Z(R)$ is a semiprime ring. 

(ii)  $Q(Z(R))\simeq\prod_{\gq\in \min   (Z(R))}Z(R)_{\gq}$ {\em is a finite direct product of fields} $Z(R)_{\gq}$:  The commutative ring $Z(R)$ is semiprime with $|\min (Z(R))|<\infty$. Hence, 
$$\CC_{Z(R)}=Z(R)\backslash \bigcup_{\gq\in \min (Z(R))}\gq$$
and the quotient  ring $Q(Z(R))$ is a semiprime commutative Artinian ring with $$\min (Q(Z(R)))=\{\CC_{Z(R)}^{-1}\gq\, | \, \gq \in\min (R) \}.$$ Hence, the ring $Q(Z(R))$ is a semisimple Artinian  ring with $\Spec (Z(R))=\min (Z(R))$, and so   the quotient ring $Q(Z(R))\simeq\prod_{\gq\in \min   (Z(R))}Z(R)_{\gq}$ is a finite direct product of fields. 

Let 
$$1=\sum_{\gq \in \min (Z(R))}e_\gq$$
be the corresponding sum of central orthogonal idempotents of the ring $Q(Z(R))$.

(iii) {\em The ring $\CC_{Z(R)}^{-1}R$ is  a semiprime ring:} Since $\CC_{Z(R)}\subseteq \CC_R$, the ring $R$  is a subring of $\CC_{Z(R)}^{-1}R$. Suppose that $\ga$ is a nonzero nilpotent ideal of the ring $\CC_{Z(R)}^{-1}R$. Then the intersection $\ga'=R\cap \ga$ is a nonzero nilpotent ideal of $R$ (since the left $R$-module $R$ is an essential submodule of $\CC_{Z(R)}^{-1}R$), a contradiction, and the statement (iii) follows.

(iv) {\em The ring  $Q_{l,cl}(\CC_{Z(R)}^{-1}R )\simeq \prod_{\gq\in \min   (Z(R))}Q_{l,cl}(R_{\gq})$ is a semisimple Artinian ring}: 
By the assumption, $\CC_{Z(R)}\subseteq \CC_R$. Now,
\begin{eqnarray*}
R&\subseteq & \CC_{Z(R)}^{-1}R \simeq  \CC_{Z(R)}^{-1}Z(R)\t_{Z(R)}R \simeq Q(Z(R))\t_{Z(R)}R\simeq \prod_{\gq\in \min (Z(R))}Q(Z(R))e_\gq R \\
&\simeq &\prod_{\gq\in \min (Z(R))}Z(R)_\gq e_\gq R \simeq \prod_{\gq\in \min (Z(R))}Z(R)_\gq \t_{Z(R)} R \simeq \prod_{\gq\in \min (Z(R))} R_\gq. 
\end{eqnarray*}
By the statement (iii), the ring $\CC_{Z(R)}^{-1}R$ is a semiprime ring then so are its direct product  components $R_\gq$. 
Since  the rings $R_\gq$ are semiprime left Goldie rings, the rings $Q_{l,cl}( R_\gq)$ are semisimple Artinian rings, by Goldie's Theorem. Hence, the finite direct product of these rings,
$\prod_{\gq\in \min (Z(R))} Q_{l,cl}(R_\gq)$, is also a semisimple Artinian ring. Finally, 
 $$
 \prod_{\gq\in \min (Z(R))} Q_{l,cl}(R_\gq)
 \simeq Q_{l,cl}\bigg(\prod_{\gq\in \min (Z(R))} R_\gq\bigg)\simeq Q_{l,cl}(\CC_{Z(R)}^{-1}R)\\
$$
and the statement (iv) follows.

(v) $S:=R\cap \CC_{ \CC_{Z(R)}^{-1}R} \in \Den_l(R,0)$ {\em and} $S^{-1}R\simeq Q_{l,cl}(\CC_{Z(R)}^{-1}R)$: The statement (v) is a particular case of  \cite[Lemma 3.3.(1)]{larglquot}.

(vi) $Q_l(R)\simeq Q_{l,cl}(\CC_{Z(R)}^{-1}R)\simeq \prod_{\gq\in \min (Z(R))} Q_{l,cl}(R_\gq)$ {\em is a semisimple Artinian ring}: Notice that $S_l(R)\supseteq S$ since $S=R\cap \CC_{ \CC_{Z(R)}^{-1}R} \in \Den_l(R,0)$ (the statement (v)). Then $S^{-1}R\subseteq Q_l(R)$. In fact, the equality holds since the ring $S^{-1}R\simeq 
 Q_{l,cl}(\CC_{Z(R)}^{-1}R)$ is a semisimple Artinian ring, by the statement (iv).

(vii) $Q_{l,cl}(R)\simeq Q_l(R)$: The statement (vii) follows from the statement (vi) and Theorem \ref{5Jul10}.
 
  $(1\Rightarrow 2)$ 
(i) {\em The centre $Z(R)$ is a semiprime ring}: Suppose that the  centre $Z(R)$ is not  a semiprime ring. Then it contains a nonzero nilpotent ideal, say $\ga$. Then $\CC_R^{-1}\ga$ is a nonzero nilpotent ideal of the semisimple  Artinian ring $Q_{l,cl}(R)$, a contradiction. Therefore, the centre $Z(R)$ is a semiprime ring. 

(ii) $|\min (Z(R))|<\infty$: Suppose that $|\min (Z(R))|=\infty$. Then we can choose an infinite set of minimal prime ideals of $Z(R)$, say $\gp_1, \ldots , \gp_n,\ldots $. For each natural number $n\geq 1$,  the set $$S_n
:=Z(R)\backslash \bigcup_{i=1}^n\gp_i$$ is a  multiplicative subset of $Z(R)$ such that the localization $S_n^{-1}Z(R)$ is a semiprime Artinian commutative ring. Hence, it is a direct product of    fields
$$ S_n^{-1}Z(R)\simeq\prod_{i=1}^n Z(R)_{\gp_i}.$$           
Let $1=\sum_{j=1}^ne_{ij}$ be the corresponding sum of orthogonal idempotents. Clearly, $S_n\in \Den (Q_{l,cl}(R))$ and the left uniform dimension of the ring $S_n^{-1}Q_{l,cl}(R)$ does not exceed the left uniform dimension of the ring $Q_{l,cl}(R)$ which is a natural number. On the other hand, in view of the equality $1=\sum_{j=1}^ne_{ij}$,  the left uniform dimension of the ring $S_n^{-1}Q_{l,cl}(R)$ cannot be smaller than $n$, a contradiction. Therefore, $|\min (Z(R))|<\infty$.

(iii) {\em For each $\gq\in   \min (Z(R))$, $Q_{l,cl}(R_\gq)\simeq Q_l(R_\gq)\simeq Q_{l,cl}(R)_\gq$ is a semisimple Artinian ring and 
 the ring $R_\gq$ is a left Goldie ring}: 
Notice that $$S_\gq:=R\backslash \gq\subseteq Z(R)\subseteq R\subseteq  Q_{l,cl}(R).$$ The set $S_\gq $ is a denominator set of the semisimple Artinian ring $Q_{l,cl}(R)$. Hence, the ring 
$$Q_{l,cl}(R)_\gq \simeq Q_{l,cl}(R)/\ass (S_\gq)$$ is a semisimple Artinian ring where $\ass (S_\gq):=\{ q\in Q_{l,cl}(R)\, | \, sq=0$ for some $s\in S_\gq \}$. Since $S_\gq \subseteq Z(R)\subseteq Z(Q_{l,cl}(R))$, 
$$Q_{l,cl}(R)_\gq \simeq S_\gq^{-1}\CC_R^{-1}R\simeq  \s_\gq (\CC_R)^{-1}\Big(S_\gq^{-1}R\Big)\simeq  \s_\gq (\CC_R)^{-1}R_\gq $$ where $\s_\gq : R\mapsto R_\gq$, $r\mapsto \frac{r}{1}$ and  $\s_\gq (\CC_R)\in \Den_l(R_\gq,0)$. Now the statement (iii) follows from Proposition \ref{A24Sep23}.(1). 
\end{proof}

For a ring $R$  such that its centre $Z(R)=\prod_{i=1}^nK_i$ is a finite direct product of fields $K_i$, Corollary \ref{B23Sep23} is a criterion for the ring $Q_{l,cl}(R)$ being  a semisimple Artinian ring.

\begin{corollary}\label{B23Sep23}
Let $R$ be a ring such that its centre $Z(R)=\prod_{i=1}^nK_i$ is a finite direct product of fields $K_i$ and $1=\sum_{i=1}^ne_i$ be the corresponding sum of orthogonal idempotents.   Then the  following statements are equivalent:
\begin{enumerate}

\item The ring $Q_{l,cl}(R)$ is a semisimple Artinian ring.

\item The ring $R$ is a semiprime ring and the factor ring $R/Re_i$ is a left Goldie ring for all $i=1, \ldots , n$.

\end{enumerate}
If one of the equivalent conditions holds then $Q_{l,cl}(R)\simeq\prod_{i=1}^nQ_{l,cl}(R/R(1-e_i))$, $$Q_{l,cl}(R/R(1-e_i))\simeq Q_{l,cl}(R)/Q_{l,cl}(R)(1-e_i)\;\;  {\rm for\; all}\;\;i=1, \ldots , n,$$ and the map $\min (R)\ra \min (Z(R))$, $\gp\mapsto \gp^{\rm r}:=\gp\cap Z(R)$ is a surjection.
\end{corollary}

\begin{proof} By the assumption,  $Z(R)=\prod_{i=1}^nK_i$ is a finite direct product of fields $K_i$. Hence $\min (Z(R))=\{ \gq_i:=Z(R)(1-e_i)\, | \, i=1, \ldots , n\}=\{ \l =(\l_1, \ldots , \l_n)\, | \, \l_j\in K_j$ for $j=1, \ldots , n$ and $\l_i=0\}$. 

In view of Theorem \ref{VVB-23Sep2023}, to prove the equivalence of statements 1 and 2,  we have to show that $\CC_{Z(R)}\subseteq \CC_R$ and  for each $i=1, \ldots , n$, the ring $R_{\gq_i}$ is a left Goldie ring.

(i) $\CC_{Z(R)}\subseteq \CC_R$: Since $Z(R)=\prod_{i=1}^nK_i$ is a finite direct product of fields $K_i$, $$\CC_{Z(R)}=\prod_{i=1}^nK_i^\times=\CC_{Z(R)}^\times ,$$ and so $\CC_{Z(R)}\subseteq R^\times \subseteq \CC_R$.

(ii) {\em For each $i=1, \ldots , n$, the ring $Z(R)_{\gq_i }\simeq Z(R)/Z(R)(1-e_i)\simeq K_i$}: Clearly, $Z(R)/Z(R)(1-e_i)\simeq K_i$. Since $e_i (1-e_i)=0$ and $\gq_i=Z(R)(1-e_i)$, $e_i\not\in \gq_i$. It follows from the equality $e_i \gq_i=0$ that 
 $$Z(R)_{\gq_i}\simeq \Big( Z(R)/\gq_i\Big)_{\gq_i}\simeq \Big(K_i\Big)_{\gq_i}=K_i.$$

(iii) {\em For each $i=1, \ldots , n$,  $R_{\gq_i }\simeq R/R(1-e_i)$}: By the statement (ii),
$$R_{\gq_i }\simeq Z(R)_{\gq_i}\t_{Z(R)}R=\Big(Z(R)/ Z(R)(1-e_i)\Big) \t_{Z(R)}R\simeq R/R(1-e_i).$$

(iv) {\em For each $\gq\in   \min (Z(R))$, the ring $R_\gq$ is a left Goldie ring}: The statement (iv) follows from the statement (iii) and the equality $\min (Z(R))=\{ \gq_i=Z(R)(1-e_i)\, | \, i=1, \ldots , n\}$. 

By the statements (i)-(iv) and Theorem \ref{VVB-23Sep2023}, statements 1 and 2 are equivalent. 

Suppose that statements 1 and 2 are equivalent. Then by Theorem \ref{VVB-23Sep2023}, $$Q_{l,cl}(R)\simeq\prod_{\gq\in \min   (Z(R))}Q_{l,cl}(R_{\gq}),$$ $Q_{l,cl}(R_{\gq})\simeq Q_{l,cl}(R)_{\gq}$ for all $\gq\in \min (Z(R))$, and the map $\min (R)\ra \min (Z(R))$, $\gp\mapsto \gp^{\rm r}:=\gp\cap Z(R)$ is a surjection.
 Notice that $R_{\gq_i}\simeq R/R(1-e_i)$, and so 
 $Q_{l,cl}(R)\simeq\prod_{i=1}^nQ_{l,cl}(R/R(1-e_i))$. 
 Now, for  every $i=1, \ldots , n$,
 \begin{eqnarray*}
Q_{l,cl}(R/R(1-e_i))&=&Q_{l,cl}(R_{\gq_i})\simeq Q_{l,cl}(R)_{\gq_i}\simeq Z(R)_{\gq_i}\t_{Z(R)}Q_{l,cl}(R) \\
&\simeq &\Big(Z(R)/ Z(R)(1-e_i)\Big) \t_{Z(R)}Q_{l,cl}(R)
\simeq  Q_{l,cl}(R)/Q_{l,cl}(R)(1-e_i).
\end{eqnarray*}
The proof of the corollary is complete.
\end{proof}

{\bf Licence.} For the purpose of open access, the author has applied a Creative Commons Attribution (CC BY) licence to any Author Accepted Manuscript version arising from this submission.

{\bf Disclosure statement.} No potential conflict of interest was reported by the author.

{\bf Data availability statement.} Data sharing not applicable – no new data generated.

\small{

School of Mathematics and Statistics

University of Sheffield

Hicks Building

Sheffield S3 7RH

UK

email: v.bavula@sheffield.ac.uk}


\begin{thebibliography}{99}

\bibitem{Bav-intdif-mIn} V. V. Bavula, The algebra of integro-differential operators on a
polynomial algebra, {\em Journal of the London Math. Soc.}, {\bf 83} (2011) no. 2, 517-543. 

 
 \bibitem{Bav-intdifline} V. V. Bavula,  The algebra of integro-differential operators on an affine line and its modules, {\em J. Pure Appl. Algebra} {\bf 217} (2013)  495-529. (Arxiv:math.RA: 1011.2997).
 
\bibitem{NewCrit-S-Simpl-lQuot} V. V. Bavula,  New criteria for a ring to have a semisimple left quotient ring,  {\em Journal of Alg. and its Appl.}, {\bf 14} (2015) no. 6, 1550090, 28pp
 
 

\bibitem{larglquot} V. V. Bavula,  The largest left quotient ring of a ring, {\it Comm. Algebra}  {\bf 44} (2016)
 no. 8, 3219-3261. (Arxiv:math.RA:1101.5107).
  
 
 
\bibitem{MinPrimeLocRings} V. V. Bavula, The minimal primes of localizations of rings, submitted. 
    
    
 
 
  
\bibitem{Goldie-PLMS-1958} A. W.  Goldie, The structure of prime rings under ascending chain conditions. {\it Proc. London Math. Soc.}  (3) {\bf 8} (1958) 589--608.


\bibitem{Goldie-PLMS-1960} A. W.  Goldie,  Semi-prime rings with maximum condition. {\it Proc. London Math. Soc.}  (3) {\bf 10} (1960) 201--220.




\bibitem{Lesieur-Croisot-1959} L. Lesieur and R. Croisot, Sur les anneaux premiers noeth\'{e}riens \`{a} gauche. {\it Ann. Sci. \'{E}cole Norm. Sup.} (3) {\bf 76} (1959) 161--183.
 

Wiley, Chichester,  1987.









\end{thebibliography}
\end{document}